\documentclass[final]{siamart0516}

\usepackage{lipsum}
\usepackage{amsfonts}
\usepackage{graphicx}
\usepackage{epstopdf}
\usepackage{algorithmic}
\usepackage{hyperref}

\ifpdf
  \DeclareGraphicsExtensions{.eps,.pdf,.png,.jpg}
\else
  \DeclareGraphicsExtensions{.eps}
\fi

\newcommand{\TheTitle}{On Best-Response Dynamics in Potential Games}
\newcommand{\TheAuthors}{B. Swenson, R. Murray, and S. Kar}

\headers{\TheTitle}{\TheAuthors}

\title{{\TheTitle}}

\usepackage{amsopn}

\usepackage{amssymb}
\usepackage{cite}
\usepackage[multiple]{footmisc}
\usepackage{mathtools}
\usepackage{enumitem}
\usepackage{graphicx}
\usepackage{caption}
\usepackage{subcaption}
\usepackage{float}

\newcommand{\Gr}{\mbox{\textup{Graph}}}

\newcommand{\mbb}{\mathbb}

\newcommand{\R}{\mathbb{R}}
\newcommand{\N}{\mathbb{N}}

\newcommand{\dx}{\,dx}
\newcommand{\ds}{\,ds}

\newcommand{\spt}{\mbox{\textup{spt}}}

\newcommand{\carr}{\mbox{\textup{carr}}}

\newcommand{\diam}{\mbox{\textup{diam\,}}}

\newcommand{\BRtilde}{\stackrel{\sim}{\smash{\mbox{BR}}\rule{0pt}{1.2ex}}}
\newcommand{\actionik}{y_{i}^{k}}

\newcommand{\actionikone}{y_{i}^{k+1}}
\newcommand{\actionione}{y_{i}^{1}}

\newcommand{\BR}{\mbox{\textup{BR}}}
\newcommand{\FP}{\textup{BRD}}
\newcommand{\cl}{\mbox{\textup{cl\,}}}
\newcommand{\ri}{\mbox{\textup{ri\,}}}

\newcommand{\ddt}{\frac{d}{dt}}

\DeclareRobustCommand{\rchi}{{\mathpalette\irchi\relax}}
\newcommand{\irchi}[2]{\raisebox{\depth}{$#1\chi$}} 

   \def\vx{{\bf x}}
 \def\vz{{\bf z}}

   \def\vH{{\bf H}}
\def\vI{{\bf I}} \def\vJ{{\bf J}}  
   \def\vP{{\bf P}}

\def\calD{\mathcal{D}}  
 \def\calH{\mathcal{H}} \def\calI{\mathcal{I}}
  \def\calL{\mathcal{L}}
 \def\calN{\mathcal{N}} 
\def\calP{\mathcal{P}}

 \def\calZ{\mathcal{Z}}

\theoremstyle{remark}
\newtheorem{remark}[theorem]{Remark}
\newtheorem{example}[theorem]{Example}
\newtheorem{mydef}[theorem]{Definition}
\newtheorem{property}{Property}

\ifpdf
\hypersetup{
  pdftitle={\TheTitle},
  pdfauthor={\TheAuthors}
}
\fi

\author{Brian Swenson\footnotemark[2]\,, \hspace{2pt}Ryan Murray\footnotemark[3]\,,  and  Soummya Kar\footnotemark[2]}

\begin{document}
\maketitle

\renewcommand{\thefootnote}{\fnsymbol{footnote}}

\footnotetext[2]{Department of Electrical and Computer Engineering,
Carnegie Mellon University, Pittsburgh, PA, USA (brianswe@ece.cmu.edu, soummyak@ece.cmu.edu).}
\footnotetext[3]{Department of Mathematics, Pennsylvania State University, State College, PA, USA (rwm22@psu.edu).\\ We note that a preliminary version of the work on the rate of convergence of BR dynamics in Section \ref{sec_conv_rate} was presented at the Allerton conference on communication, control, and computing \cite{swenson2017convRate}.}

\renewcommand{\thefootnote}{\arabic{footnote}}

\begin{abstract}
The paper studies the convergence properties of (continuous) best-response dynamics from game theory. Despite their fundamental role in game theory, best-response dynamics are poorly understood in many games of interest due to the discontinuous, set-valued nature of the best-response map. The paper focuses on elucidating several important properties of best-response dynamics in the class of multi-agent games known as potential games---a class of games with fundamental importance in multi-agent systems and distributed control.
It is shown that in almost every potential game and for almost every initial condition, the best-response dynamics (i) have a unique solution, (ii) converge to pure-strategy Nash equilibria, and (iii) converge at an exponential rate.
\end{abstract}



\section{Introduction} \label{sec_intro}
A Nash equilibrium (NE) is a solution concept for multi-player games in which no player can unilaterally improve their personal utility.
Formally, a Nash equilibrium is defined as a fixed point of the best-response mapping---that is, a strategy $x^*$ is said to be a NE if
$$
x^* \in \BR(x^*),
$$
where BR denotes the (set-valued) best-response mapping (see Section \ref{sec_prelims} for a formal definition).

A question of fundamental interest is, given the opportunity to interact, how might a group of players adaptively learn to play a NE strategy over time?
In response, it is natural to consider the dynamical system induced by the best response mapping itself:\footnote{Since the map $\BR$ is set-valued (see Section \ref{sec_notation}), the dynamical system \eqref{def_FP_autonomous} is a differential \emph{inclusion} rather than a differential equation.
However, as we will see later in the paper, in potential games the $\BR$ map can generally be shown to be almost-everywhere single-valued along solution curves of \eqref{def_FP_autonomous} (see Remark \ref{remark_BR_diff_equation}). Thus, for the intents and purposes of this paper, it is relatively safe to think of \eqref{def_FP_autonomous} as a differential equation $\dot{\vx} =\BR(\vx) - \vx$ with discontinuous right-hand side.}
\begin{equation} \label{def_FP_autonomous}
\dot{\vx} \in \BR(\vx) - \vx.
\end{equation}
By definition, the set of NE coincide with the equilibrium points of these dynamics.
In the literature, the learning procedure \eqref{def_FP_autonomous} is generally referred to as \emph{best-response dynamics} (BR dynamics) \cite{gilboa1991social,matsui1992best,hofbauer1995stability,hopkins1999note,
hofbauer2003evolutionary,cortes2015self}.\footnote{For reasons soon to become clear, the system \eqref{def_FP_autonomous} is sometimes also referred to as (continuous-time) \emph{fictitious play} \cite{harris1998rate,shamma2004unified,krishna1998convergence}.
We will favor the term ``BR dynamics''.}

BR dynamics are fundamental to game theory and have been studied in numerous works including \cite{gilboa1991social,matsui1992best,hofbauer1998evolutionary,hofbauer1995stability,
hopkins1999note,cortes2015self,harris1998rate,benaim2005stochastic,leslie2006generalised,krishna1998convergence}.
These dynamics model various forms of learning in games. From the perspective of evolutionary learning, \eqref{def_FP_autonomous} can be seen as modeling adaptation in a large population when some small fraction of the population revises their strategy as a best response to the current population strategy \cite{hofbauer2003evolutionary}. Another interpretation is to consider games with a finite number of players and suppose that each player continuously adapts their strategy according to the dynamics (1).
From this perspective, the dynamics \eqref{def_FP_autonomous} are closely related to a number of discrete-time learning processes. For example, the popular fictitious play (FP) algorithm---a canonical algorithm that serves as a prototype for many others---is merely an Euler discretization of \eqref{def_FP_autonomous}.\footnote{Hence the name ``continuous-time fictitious play'' for \eqref{def_FP_autonomous}.} Accordingly, many asymptotic properties of FP can be determined by studying the asymptotic properties of the system \eqref{def_FP_autonomous} (this approach is formalized in \cite{benaim2005stochastic}).
As another example, \cite{leslie2006generalised} studies a payoff-based algorithm in which players estimate expected payoffs and tend, with high probability, to pick best response actions; the underlying dynamics can again be shown to be governed by the differential inclusion \eqref{def_FP_autonomous}.
There are numerous other learning algorithms \cite{long2007non,mohsenian2010autonomous,shamma2005dynamic,coucheney2014penalty,swenson2017SSFP,swenson2012ECFP} (to cite a few recent examples) that rely on best-response adaptation whose underlying dynamics are heavily influenced by \eqref{def_FP_autonomous}.
Using the framework of stochastic approximation \cite{benaim2005stochastic,kushner2003stochastic,borkar2008stochastic,
swenson2017robustness}, the relationship between such discrete and continuous dynamics can be made rigorous.

The BR dynamics are also closely related to other popular learning dynamics including the replicator dynamics \cite{hofbauer2009time} and positive definite adaptive dynamics \cite{hopkins1999note}. While these dynamics do not rely explicitly on best-response adaptation, their asymptotic behavior is often similar to that of BR dynamics \cite{hofbauer2009time,hopkins1999note}.

Despite the general importance of BR dynamics in game theory, they are not well understood in many games of interest. In large part, this is due to the fact that, since \eqref{def_FP_autonomous} is a differential \emph{inclusion} rather than a classical differential equation, it is often difficult to analyze using classical techniques. In this work our objective will be to elucidate several fundamental properties of \eqref{def_FP_autonomous} within the important class of multi-agent games known as potential games \cite{Mond96}.\footnote{We will focus here on potential games with a finite number of actions, saving continuous potential games for a future work.}

In a potential game, there exists an underlying potential function which all players implicitly seek to maximize. Such games are fundamentally cooperative in nature (all players benefit by maximizing the potential) and have a tremendous number of applications in both economics \cite{Mond96,sandholm2001potential} and engineering \cite{marden-connections,li2013designing,cheng2014finite,scutari2006potential,
buzzi2012potential,saad2012game,soto2009distributed}, where
game-theoretic learning dynamics such as \eqref{def_FP_autonomous} are commonly used as mechanisms for distributed control of multi-agent systems.
Along with so-called \emph{harmonic games} (which are fundamentally adversarial in nature), potential games may be seen as one of the basic building blocks of general $N$-player games \cite{candogan-potential-decomposition}.

In finite potential games there are several important properties of \eqref{def_FP_autonomous} that are not well understood. For example
\begin{itemize}
\item It is not known if solutions of \eqref{def_FP_autonomous} are generically well posed. More precisely, being a differential inclusion, it is known that solutions of \eqref{def_FP_autonomous} may lack uniqueness for some initial conditions. But it is not understood if non-uniqueness of solutions is typical, or if it is somehow exceptional. E.g., it may be the case that solutions are unique from almost every initial condition, and hence \eqref{def_FP_autonomous} is, in fact, generically well posed in potential games. This has been speculated \cite{hofbauer1995stability}, but has never been shown.
\item There are no convergence rate estimates for BR dynamics in potential games. 
    It has been conjectured that the rate of convergence of \eqref{def_FP_autonomous} is exponential in potential games (\hspace{-.01em}\cite{harris1998rate}, Conjecture 25). However, this has never been shown.
\end{itemize}
Regarding the second point, we remark that, even outside the class of BR-based learning dynamics, convergence rate estimates for game-theoretic learning dynamics are relatively scarce.
Thus, given the practical importance of convergence rates in applications, there is strong motivation for establishing rigorous convergence rate estimates for learning dynamics in general. This is particularly relevant in engineering applications with large numbers of players.

Underlying both of the above issues is a single common problem. The set of NE may be subdivided into pure-strategy (deterministic) NE and mixed-strategy (probabilistic) NE. Mixed-strategy NE tend to be highly problematic for a number of reasons \cite{jordan1993}.
Being a differential inclusion, trajectories of \eqref{def_FP_autonomous} can reach a mixed equilibrium in finite time. In a potential game, once a mixed equilibrium has been reached, a trajectory can rest there for an arbitrary length of time before moving elsewhere. This is both the cause of non-uniqueness of solutions and the principal reason why it is impossible to establish general convergence rate estimates which hold at all points.

In addition to hindering our theoretical understanding of BR dynamics, mixed equilibria are also highly problematic from a more application-oriented perspective. Mixed equilibria are nondeterministic, have suboptimal expected utility, and do not always have clear physical meaning \cite{arslan2007autonomous}; consequently, in engineering applications, practitioners often prefer algorithms that are guaranteed to converge to a pure-strategy NE \cite{marden-payoff,marden2009overcoming,arslan2007autonomous,li2013designing}.

It has been speculated that, in potential games (or at least, outside of zero-sum games), BR dynamics should rarely converge to mixed equilibria \cite{krishna1998convergence,hofbauer1995stability,arslan2007autonomous}, and thus the aforementioned issues should rarely arise in practice. However, rigorous proof of such a result has not been available.

In this paper we attempt to shed some light on these issues. Informally speaking, we will show that in potential games (i) BR dynamics almost never converge to mixed equilibria, (ii) solutions of \eqref{def_FP_autonomous} are almost always unique, and (iii) the rate of convergence of \eqref{def_FP_autonomous} is almost always exponential.

The study of these issues is complicated by the fact that one can easily construct counterexamples of potential games where BR dynamics behave poorly. In order to eliminate such cases, we will focus on games that are ``regular'' as introduced by Harsanyi \cite{harsanyi1973oddness}. Regular games (see Section \ref{sec_reg_games}) have been studied extensively in the literature \cite{van1991stability}; they are highly robust and simple to work with, and, as we will see below, are ideally suited to studying BR dynamics.

The linchpin in addressing all of the issues noted above is gaining a rigorous understanding of the question of convergence to mixed equilibria. The main result of the paper is the following theorem, which shows that convergence to such equilibria is in fact exceptional.
Before stating the theorem we note that, unless explicitly specified otherwise, throughout the paper we use the term ``potential game'' broadly to mean a weighted potential game \cite{Mond01} (which includes exact potential games and games with identical payoffs as special cases).
\begin{theorem} \label{thrm_main_result}
Suppose $\Gamma$ is a regular potential game. Then from almost every initial condition, solutions of \eqref{def_FP_autonomous} converge to a pure-strategy Nash equilibrium.
\end{theorem}

In a companion paper \cite{swenson2017regular} we show that ``most'' potential games are regular. In particular, Theorem 1 of \cite{swenson2017regular} shows that (i) almost every weighted potential game is regular, (ii) almost every exact potential game is regular, and (iii) almost every game with identical payoffs is regular.\footnote{When we say that a property holds for almost every game of a certain class, we mean that the subset of games in that class where the property fails to hold has (appropriately dimensioned) Lebesgue measure zero. See \cite{swenson2017regular} for more details. Note that the class of games with identical payoffs is a measure-zero subset within the class of exact potential games which, in turn, is a measure-zero subset within the class of weighted potential games. Thus, generic regularity in a superclass does not imply generic regularity in a subclass.}
Thus, as an immediate consequence of Theorem \ref{thrm_main_result} above and Theorem 1 of \cite{swenson2017regular} we get the following result:
\begin{theorem}
In almost every weighted potential game, almost every exact potential game, and almost every game with identical payoffs, solutions of \eqref{def_FP_autonomous} converge to a pure-strategy Nash equilibrium from almost every initial condition.
\end{theorem}
Moreover, since pure equilibria are highly stable in regular potential games (they are strict \cite{van1991stability}), Theorem \ref{thrm_main_result} also implies that BR dynamics generically converge to equilibria that are stable both to perturbations in strategies and payoffs.
As a byproduct of the proof of Theorem \ref{thrm_main_result} we will get the following result regarding uniqueness of solutions of \eqref{def_FP_autonomous}:\footnote{The notion of solutions considered in the paper is defined in Section \ref{sec_BR_dynamics_def}.}
\begin{proposition} \label{prop_uniqueness}
Suppose $\Gamma$ is a potential game. Then for almost every initial condition, solutions of \eqref{def_FP_autonomous} are unique.
\end{proposition}
Finally, as a simple application of Theorem \ref{thrm_main_result} we will prove the following result:
\begin{theorem} \label{thrm_conv_rate}
Suppose $\Gamma$ is a potential game. Then for almost every initial condition, solutions of \eqref{def_FP_autonomous} converge to the set of NE at an exponential rate.
\end{theorem}
We remark that this resolves the Harris conjecture (\hspace{-.01em}\cite{harris1998rate}, Conjecture 25) on the rate of convergence of continuous-time fictitious play in weighted potential games.\footnote{Harris \cite{harris1998rate} showed that the rate of convergence of \eqref{def_FP_autonomous} is exponential in zero-sum games and conjectured that the rate of convergence is exponential in (weighted) potential games. However, due to the problems arising from mixed equilibria in potential games, \cite{harris1998rate} does not attempt to prove this conjecture. Our main result allows us to handle the problems arising due to mixed equilibria, and consequently allows for an easy proof of the exponential rate of convergence as conjectured in \cite{harris1998rate}.}\footnote{A preliminary version of the convergence rate estimate in Theorem \ref{thrm_conv_rate} (see also Proposition \ref{prop_FP_conv_rate1}) can be found in an earlier conference version of this work \cite{swenson2017convRate}.}

We also remark that the question of (non-) convergence of BR dynamics to mixed equilibria was previously considered\footnote{More precisely, \cite{krishna1998convergence} considers a variant of \eqref{def_FP_autonomous}, equivalent after a time change.}
in \cite{krishna1998convergence} for the case of two-player games where it was shown that BR dynamics ``almost never converge cyclically to a mixed-strategy equilibrium in which both players use more than two pure strategies.'' In contrast, in this paper we consider potential games of arbitrary size and prove (generic) non-convergence to all mixed equilibria.

The remainder of the paper is organized as follows. In Sections \ref{sec_proof_strategy}--\ref{sec_comparison} we outline our high-level strategy for proving Theorem \ref{thrm_main_result} and compare with classical techniques.  Sections \ref{sec_notation}--\ref{sec_BR_dynamics_def} set up notation. Section \ref{sec_example} gives a simple two-player example illustrating the fundamental problems arising in BR dynamics in potential games.
Section \ref{sec_reg_games} introduces regular potential games.
Section \ref{sec_inequalities} establishes the two key inequalities used to prove Theorem \ref{thrm_main_result}. Section \ref{sec_proof_main_result} gives the proof of Theorem \ref{thrm_main_result}. Section \ref{sec_uniqueness} proves uniqueness of solutions (Proposition \ref{prop_uniqueness}). Section \ref{sec_conv_rate} proves the exponential convergence rate estimate (Theorem \ref{thrm_conv_rate}).
Section \ref{sec_conclusion} concludes the paper.

\subsection{Proof Strategy} \label{sec_proof_strategy}
The basic strategy is to leverage two noteworthy properties satisfied in \emph{regular} potential games:\\
\begin{enumerate}
\item  The BR dynamics cannot
cannot concentrate volume in finite time, meaning the flow induced by \eqref{def_FP_autonomous} cannot map a set of positive (Lebesgue) measure to a set of zero measure in finite time (see Section \ref{sec_finite_time_convergence}).\\
\item  In a neighborhood of an \emph{interior} Nash equilibrium (i.e., a \emph{completely mixed equilibrium}), the magnitude of the time derivative of the potential along paths grows linearly in the distance to the Nash equilibrium, while the value of the potential varies only quadratically; that is,
\[
\frac{d}{dt} U(\vx(t)) \geq d(\vx(t),x^*) \geq \sqrt{|U(\vx(t)) - U(x^*)|},
\]
where $U$ denotes the potential function and $x^*$ is the equilibrium point.\\
\end{enumerate}

Using Markov's inequality, property 2 immediately implies that if a path converges to an interior NE then it must do so in finite time (see Section \ref{sec_infinite_time_conv}). Hence, properties 1 and 2 together imply that the set of points from which BR dynamics converge to an interior NE must have Lebesgue measure zero.

In order to handle mixed NE that are \emph{not} in the interior of the strategy space (i.e., \emph{incompletely mixed equilibria}), we consider a projection that maps incompletely mixed equilibria to the \emph{interior} of the strategy space of a lower dimensional game.
Using the techniques described above we are then able to handle completely and incompletely mixed equilibria in a unified manner.

In particular, since the number of NE strategies is finite in regular potential games, we see that the set of points from which BR dynamics converge to the set of mixed-strategy NE has Lebesgue measure zero in any such game. Since any solution of \eqref{def_FP_autonomous} must converge to a NE \cite{benaim2005stochastic}, this implies Theorem \ref{thrm_main_result}.

Properties 1 and 2 hold as long the equilibrium $x^*$ is regular. In a companion paper \cite{swenson2017regular} we show that, in almost all potential games, all equilibria are regular.


\subsection{Comparison with Classical Techniques} \label{sec_comparison}

Given a classical ODE, one can prove that an equilibrium point may only be reached from a set of measure zero by studying the linearized dynamics at the equilibrium point.
Assuming all eigenvalues associated with the linearized system are non-zero, the dimension of the stable manifold (i.e., the set of initial conditions from which the equilibrium can be reached) is equal to the dimension of the stable eigenspace of the linearized system \cite{Chicone_ODE}. Hence, to prove that an equilibrium can only be reached from a set of measure zero, it is sufficient to prove that at least one eigenvalue of the linearized system lies in the right half plane.

In BR dynamics, the vector field is discontinuous---hence, it is not possible to linearize around an equilibrium point, and such classical techniques cannot be directly applied.
However, the gradient field of the potential function is closely linked to the BR-dynamics vector field (i.e., the vector field $\BR(x) - x$; see Lemma \ref{lemma_IR1} for more details).
Unlike the BR-dynamics vector field, the gradient field of the potential function \emph{can} be linearized. In a non-degenerate game, any completely mixed-strategy NE is a non-degenerate saddle point of the potential function. Hence, at least one eigenvalue of the linearized gradient system must lie in the right-half plane. This implies that, for the gradient dynamics of the potential function, the stable manifold associated with an equilibrium point has dimension at most $\kappa-1$, where $\kappa$ is the dimension of the strategy space.

Given the close relationship between the BR-dynamics vector field and the gradient field of the potential function, intuition suggests that for BR dynamics, each mixed equilibrium should also admit a similar low-dimensional stable manifold.

While this provides an intuitive explanation for why one might expect Theorem \ref{thrm_main_result} to hold, we did not use any such linearization arguments in the proof of this result. We found that studying the rate of potential production near mixed equilibria (e.g., as discussed in the ``proof strategy'' section above) led to shorter and simpler proofs.

\section{Preliminaries} \label{sec_prelims}
\subsection{Notation} \label{sec_notation}
A game in normal form is represented by the tuple \newline $\Gamma := (N,(Y_i,u_i)_{i=1,\ldots,N})$, where $N\in\{2,3,\ldots\}$ denotes the number of players, $Y_i=\{y_i^1,\ldots,y_i^{K_i}\}$ denotes the set of pure strategies (or actions) available to player $i$, with cardinality $K_i := |Y_i|$, and $u_i:\prod_{j=1}^N Y_j \rightarrow \mbb{R}$ denotes the utility function of player $i$. Denote by $Y:= \prod_{i=1}^N Y_i$ the set of joint pure strategies, and let $K := \prod_{i=1}^N K_i$ denote the number of joint pure strategies.

For a finite set $S$, let $\triangle(S)$ denote the set of probability distributions over $S$. For $i=1,\ldots,N$, let $\Delta_{i}:=\triangle(Y_i)$ denote the set of \emph{mixed-strategies} available to player $i$. Let $\Delta := \prod_{i=1}^N \Delta_i$ denote the set of joint mixed strategies.\footnote{It is implicitly assumed that players' mixed strategies are independent; i.e., players do not coordinate.} Let $\Delta_{-i} := \prod_{j\in\{1,\ldots,N\}\backslash\{i\}} \Delta_j$. When convenient, given a mixed strategy $\sigma=(\sigma_1,\ldots,\sigma_N)\in \Delta$, we use the notation $\sigma_{-i}$ to denote the tuple $(\sigma_j)_{j\not=i}$

Given a mixed strategy $\sigma\in\Delta$, the expected utility of player $i$ is given by
$$
U_i(\sigma_1,\ldots,\sigma_N) = \sum_{y\in Y} u_i(y)\sigma_1(y_1)\cdots \sigma_N(y_N).
$$

For $\sigma_{-i} \in \Delta_{-i}$, the best response of player $i$ is given by the set-valued function $\BR_i:\Delta_{-i}\rightrightarrows\Delta_i$,
$$
\BR_i(\sigma_{-i}):= \arg\max_{\sigma_i' \in \Delta_i} U_i(\sigma_i',\sigma_{-i}),
$$
where we use the double right arrows to indicate a set-valued function.
For $\sigma\in \Delta$ the joint best response is given by the set-valued function $\BR:\Delta\rightrightarrows\Delta$
$$
\BR(\sigma) := \BR_{1}(\sigma_{-1})\times\cdots\times \BR_{N}(\sigma_{-N}).
$$

A strategy $\sigma\in \Delta$ is said to be a Nash equilibrium (NE) if $\sigma \in \BR(\sigma)$. For convenience, we sometimes refer to a Nash equilibrium simply as an equilibrium.

We say that $\Gamma$ is a potential game (or, more precisely, a \emph{finite weighted potential game}) \cite{Mond96} if there exists a function $u:Y\rightarrow \R$ and a vector of positive weights $(w_i)_{i=1}^N$, such that $u_i(y_i',y_{-i}) - u_i(y_i'',y_{-i}) = w_i\big(u(y_i',y_{-i}) - u(y_i'',y_{-i})\big)$ for all $y_{-i} \in Y_{-i}$ and $y_i',y_i'' \in Y_i$, for all $i=1,\ldots,N$.

Let $U:\Delta\rightarrow \R$ be the multilinear extension of $u$ defined by
\begin{equation} \label{def_potential_fun1}
U(\sigma_1,\ldots,\sigma_N) = \sum_{y\in Y} u(y)\sigma_1(y_1)\cdots\sigma(y_N).
\end{equation}
The function $U$ may be seen as giving the expected value of $u$ under the mixed strategy $\sigma$. We refer to $U$ as the \emph{potential function} and to $u$ as the \emph{pure form of the potential function}.


Using the definitions of $U_i$ and $U$ it is straightforward to verify that
$$
\BR_i(\sigma_{-i}) := \arg\max_{\sigma_i\in\Delta_i} U_i(\sigma_i,\sigma_{-i}) = \arg\max_{\sigma_i\in\Delta_i} U(\sigma_i,\sigma_{-i}).
$$
Thus, in order to compute the best response set we only require knowledge of the potential function $U$, not necessarily the individual utility functions $(U_i)_{i=1,\ldots,N}$.


By way of notation, given a pure strategy $y_i \in Y_i$ and a mixed strategy $\sigma_{-i} \in \Delta_{-i}$, we will write $U(y_i,\sigma_{-i})$ to indicate the value of $U$ when player $i$ uses a mixed strategy placing all weight on the $y_i$ and the remaining players use the strategy $\sigma_{-i}\in \Delta_{-i}$.

Given a $\sigma_i \in \Delta_i$, let $\sigma_i^k$ denote value of the $k$-th entry in $\sigma_i$, so that $\sigma_i = (\sigma_i^k)_{k=1}^{K_i}$.
Since the potential function is linear in each $\sigma_i$, if we fix any $i=1,\ldots,N$ we may express it as
\begin{equation} \label{eq_potential_expanded_form}
U(\sigma) = \sum_{k=1}^{K_i} \sigma_i^k U(y_i^k,\sigma_{-i}).
\end{equation}


In order to study learning dynamics without being (directly) encumbered by the hyperplane constraint inherent in $\Delta_i$ we define
$$
X_i := \{x_i\in \R^{K_i-1}:~ 0\leq x_i^k\leq 1 \mbox{ for } k=1,\ldots,K_i-1, \mbox{ and } \sum_{k=1}^{K_i-1}x_i^k \leq 1\},
$$
where we use the convention that $x_i^k$ denotes the $k$-th entry in $x_i$ so that $x_i = (x_i^k)_{k=1}^{K_i-1}$.

Given $x_i\in X_i$ define the bijective mapping $T_i:X_i\rightarrow \Delta_i$ as $T_i(x_i) = \sigma_i$ for the unique $\sigma_i\in \Delta_i$ such that $\sigma_i^{k} = x_i^{k-1}$ for $k=2,\ldots,K_i$ and $\sigma_i^1 = 1-\sum_{k=1}^{K_i-1} x_i^k$. For $k=1,\ldots,K_i$ let $T_i^k$ be the $k$-th component map of $T_i$ so that $T_i = (T_i^k)_{i=1}^{K_i}$.



Let $X := X_1\times\cdots\times X_N$ and let $T:X\rightarrow \Delta$ be the bijection given by $T = T_1\times\cdots\times T_N$. In an abuse of terminology, we sometimes refer to $X$ as the \emph{mixed-strategy space} of $\Gamma$. When convenient, given an $x\in X$ we use the notation $x_{-i}$ to denote the tuple $(x_j)_{j\not= i}$. Letting $X_{-i} := \prod_{j\not = i} X_j$, we define $T_{-i}:X_{-i} \rightarrow \Delta_{-i}$ as $T_{-i} := (T_j)_{j\not= i}$. Let
\begin{equation}\label{def_kappa}
\kappa := \sum_{i=1}^N (|Y_i|-1)
\end{equation}
denote the dimension of $X$, and note that $\kappa\not= K$, where $K$, defined earlier, is the cardinality of the joint pure strategy set $Y$.

Throughout the paper we often find it convenient to work in $X$ rather than $\Delta$.
In order to keep the notation as simple as possible we overload the definitions of some symbols when the meaning can be clearly derived from the context. In particular, let $\BR_i:X_{-i} \rightrightarrows X_i$ be defined by $\BR_i(x_{-i}) := \{x_i\in X_i:~ \BR_i(\sigma_{-i}) = \sigma_i,~ \sigma_i\in \Delta_i,~ \sigma_{-i} \in \Delta_{-i},~ \sigma_i = T_i(x_i),~ \sigma_{-i} = T_{-i}(x_{-i})\}$.
Similarly, given an $x\in X$ we abuse notation and write $U(x)$ instead of $U(T(x))$.

Given a pure strategy $y_i\in Y_i$, we will write $U(y_i,x_{-i})$ to indicate the value of $U$ when player $i$ uses a mixed strategy placing all weight on the $y_i$ and the remaining players use the strategy $x_{-i}\in X_{-i}$.
Similarly, we will say $y_i^k \in \BR_i(x_{-i})$ if there exists an $x_i\in \BR_i(x_{-i})$ such that $T_i(x_i)$ places weight one on $y_i^k$.

Applying the definition of $T_i$ to \eqref{eq_potential_expanded_form} we see that $U(x)$ may also be expressed as
\begin{equation} \label{eq_potential_expanded_form2}
U(x) = \sum_{k=1}^{K_i-1}x_i^k U(y_i^{k+1},x_{-i}) + \left(1-\sum_{k=1}^{K_i-1}x_i^k\right)U(y_i^{1},x_{-i}).
\end{equation}
for any $i=1,\ldots,N$.

We use the following nomenclature to refer to strategies in $X$.
\begin{mydef} \label{def_equilibria_nomenclature}
(i) A strategy $x\in X$ is said to be \emph{pure} if $T(x)$ places all its mass on a single action tuple $y\in Y$.\\
(ii) A strategy $x\in X$ is said to be \emph{completely mixed} if $x$ is in the interior of $X$.\\
(iii) In all other cases, a strategy $x\in X$ is said to be \emph{incompletely mixed}.
\end{mydef}


\subsubsection{Other Notation}
Other notation as used throughout the paper is as follows.
\begin{itemize}
\item $\N:=\{1,2,\ldots\}$.
\item $\nabla_{x_i}U(x) := (\frac{\partial U}{\partial x_i^k}(x))_{k=1}^{K_i-1}$ gives the gradient of $U$ with respect to the strategy of player $i$ only. $\nabla U(x) := (\frac{\partial U}{\partial x_i^k}(x))_{\substack{i=1,\ldots,N\\ k=1,\ldots,K_i-1}}$ gives the full gradient of $U$.
\item Suppose $m,n,p\in \N$,  $F_i:\R^m\times \R^n \rightarrow \R$, for $i=1,\ldots,p$. Suppose further that $F: \R^m\times \R^n \rightarrow \R^p$ is given by $F(w,z) = (F_i(w,z))_{i=1,\ldots,p}$. Then the operator $D_w$ gives the Jacobian of $F$ with respect to the components of $w=(w_k)_{k=1,\ldots,m}$; that is
    $$
    D_w F(w,z) =
    \begin{pmatrix}
    \frac{\partial F_1(w,z)}{\partial w_1} \cdots \frac{\partial F_1(w,z)}{\partial w_m}\\
    \vdots \quad \ddots \quad \vdots\\
    \frac{\partial F_p(w,z)}{\partial w_1} \cdots \frac{\partial F_p(w,z)}{\partial w_m}\\
    \end{pmatrix}.
    $$
\item $A^c$ denotes the complement of a set $A$, and $\mathring{A}$ denotes the interior of $A$, and $\cl A$ denotes the closure of $A$.
\item The support of a function $f:\Omega\to\R$ is given by $\spt(f):=\{x\in \Omega:f(x)\not= 0\}$.
\item Given a function $f$, $\mathcal{D}(f)$ refers to the domain of $f$ and $\mathcal{R}(f)$ to the range of $f$.
\item $\calL^n$, $n\in \{1,2,\ldots\}$ refers to the $n$-dimensional Lebesgue measure.
\item Given an open set $\Omega \subset \R^n$, and $k\in \N\cup\{\infty\}$, $C_c^k(\Omega)$ denotes the set of $k$-times differentiable functions with compact support in $\Omega$.
\end{itemize}

\subsection{Best Response Dynamics} \label{sec_BR_dynamics_def}
We consider solutions of \eqref{def_FP_autonomous} in the following sense.
\begin{mydef}
We say that an absolutely-continuous mapping $\vx:\R\rightarrow X$ is a solution to \eqref{def_FP_autonomous} (or a \emph{best-response process}) with initial condition $x_0\in X$ if $\vx(0) = x_0$ and \eqref{def_FP_autonomous} holds for almost every $t\in \R$.
\end{mydef}
Since the right hand side of \eqref{def_FP_autonomous} is a set-valued map that is upper semi-continuous with non-empty, compact, convex values, and is locally bounded, solutions in this sense are guaranteed to exist \cite{Filippov}.


\subsection{Illustrative Example: BR Dynamics in a $2\times 2$ Game} \label{sec_example}


The following simple example illustrates several important properties of BR dynamics.
\begin{example} \label{example_FP1}
Consider the two-player two-action game with the following payoffs.
\begin{center}
\vspace{.2cm}

\renewcommand{\arraystretch}{1.2}
\begin{tabular} {r|c|c|}
\multicolumn{3}{c}{\hspace{1.5em} $A$\hspace{1.4em} $B$}\\
\cline{2-3}
$A$ & $1,~1$ & $0,~0$\\
\cline{2-3}
$B$ & $0,~0$ & $2,~2$\\
\cline{2-3}
\end{tabular}
\renewcommand{\arraystretch}{1.0}
\vspace{.4cm}

\end{center}
Since players share identical payoffs, this is a potential game. The BR-dynamics vector field for this game is illustrated in Figure \ref{fig:FP2by2_1} (where $x_i$ denotes the probability of player $i$ playing action $B$). Trajectories can only reach or converge to the mixed equilibrium $((1/3),(1/3))$ from a one-dimensional surface (stable manifold); this is illustrated in Figure \ref{fig:FP2by2_2}. Furthermore, trajectories starting on the stable manifold will reach the equilibrium in finite time.
Since uniqueness of solutions is lost once the mixed equilibrium is reached, solutions starting on this surface are not unique.
\end{example}

\begin{figure}[h]
    \centering
    \begin{minipage}{.45\textwidth}
        \includegraphics[height=.85\textwidth]{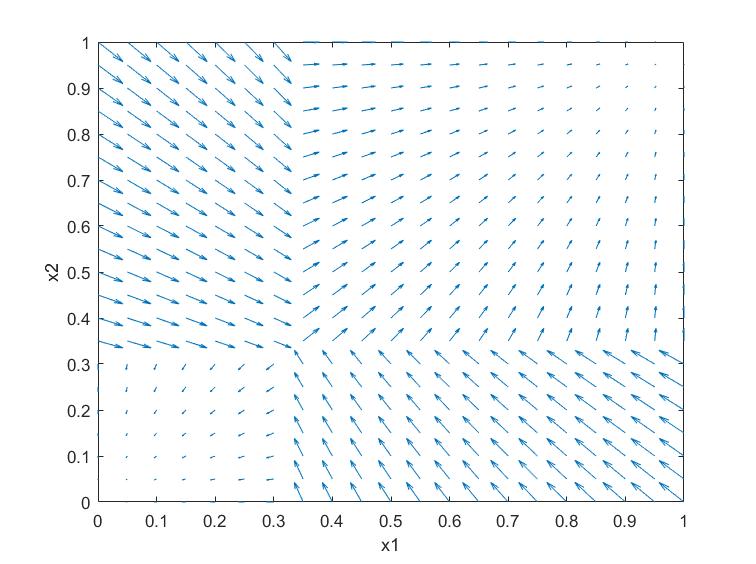}
        \caption{\small BR-dynamics vector field in Example \ref{example_FP1}.\\ $~$}
        \label{fig:FP2by2_1}
    \end{minipage}
    \hspace{.8cm}
    \begin{minipage}{.45\textwidth}
        \includegraphics[height=.85\textwidth]{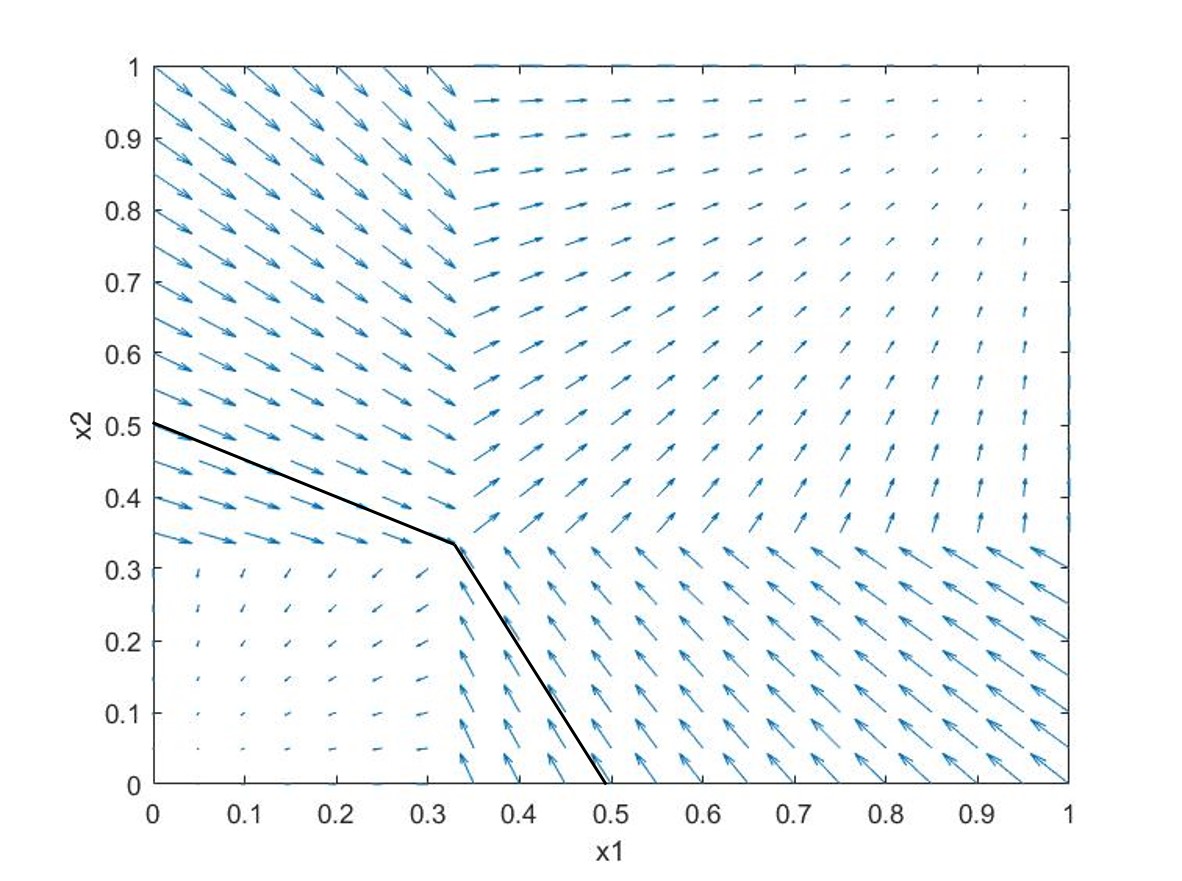}
        \caption{\small Stable manifold of the mixed equilibrium in Example \ref{example_FP1}.}
        \label{fig:FP2by2_2}
    \end{minipage}
\end{figure}




Two-player two-action games, such as the above example, possess a simple geometric structure \cite{metrick1994fictitious}, and it is relatively straightforward to see that, so long as the game is ``non-degenerate''\footnote{See, e.g., \cite{Mond01} Section 2 for a discussion of non-degenerate $2\times 2$ games.} the following properties hold for BR-dynamics in $2\times 2$ games:

\begin{property} \label{prop1}
Solution curves can only reach mixed equilibria from a set of measure zero.
\end{property}

\begin{property}\label{prop2}
Trajectories always converge to mixed equilibria in finite time.
\end{property}

\begin{property}\label{prop3}
Though solutions are not generally unique, they are unique from \emph{almost every} initial condition.
\end{property}

Our results generalize this intuition to potential games of arbitrary size.
Foremost, Theorem \ref{thrm_main_result} shows that Property \ref{prop1} holds for BR dynamics in any regular potential game.\footnote{More precisely, Since BR dynamics are guaranteed to converge to the set of NE in potential games \cite{benaim2005stochastic}, Property \ref{prop1} is equivalent to the statement of Theorem \ref{thrm_main_result}.} Proposition \ref{prop_infinite_time_convergence} shows that Property \ref{prop2} holds for BR dynamics in any regular potential game, and Proposition \ref{prop_uniqueness} shows that Property \ref{prop3} holds for BR dynamics in any regular potential game.

\section{Regular Potential Games} \label{sec_reg_games}
The notion of a regular equilibrium was introduced by Harsanyi \cite{harsanyi1973oddness}. Regular equilibria posses a variety of desirable robustness properties \cite{van1991stability}.

Being a rather stringent refinement concept, not all games possess regular equilibria. However, ``most'' games do. A game is said to be \emph{regular} if all equilibria in the game are regular. Harsanyi \cite{harsanyi1973oddness} showed that almost all $N$-player games are regular.

The set of potential games forms a low dimensional (Lebesgue-measure-zero) subspace within the space of all games. Thus,
Harsanyi's regularity result is inconclusive about the prevalence of regular games within the subset of potential games.
In a companion paper \cite{swenson2017regular}, we study this issue and show that ``most'' potential games are regular (see \cite{swenson2017regular}, Theorem 1).

In this paper we will study the behavior of BR dynamics in regular potential games. The purpose of this restriction is twofold. First, there are degenerate potential games in which BR dynamics do not converge for almost all initial conditions. Restricting attention to regular potential games ensures that the game is not degenerate in this sense.
Second, analysis of the behavior of the BR dynamics is easier near equilibria that are regular. Regularity permits us to characterize the fundamental properties of the potential function $U$ without needing to look at anything higher than second order terms in the Taylor series expansion of $U$. This substantially simplifies the analysis.

If $x^*$ is a regular equilibrium of a potential game, then the derivatives of potential function can be shown to satisfy two non-degeneracy conditions at $x^*$. The first condition deals with the gradient of the potential function at $x^*$ and is referred to as the \emph{first-order} condition; the second condition deals with the Hessian of the potential function at $x^*$ and is referred to as the \emph{second-order} condition.
These conditions, introduced in Sections \ref{sec_first_order_degeneracy}--\ref{sec_second_order_degeneracy} below, will be crucial in the subsequent analysis.


\subsection{First-Order Degeneracy} \label{sec_first_order_degeneracy}
Let $\Gamma$ be a potential game with potential function $U$. Following Harsanyi \cite{harsanyi1973oddness}, we will define the carrier set of an element $x\in X$, a natural modification of a support set to the present context.
For $x_i \in X_i$
let
$$
\carr_i(x_i) := \spt( T_i(x_i)) \subseteq Y_i
$$
and for $x = (x_1,\ldots,x_N)\in X$ let $\carr(x) := \carr_1(x_1)\cup\cdots\cup\carr_N(x_N)$.

Let $C=C_1\cup\cdots\cup C_N$, where for each $i=1,\ldots,N$, $C_i$ is a nonempty subset of $Y_i$. We say that $C$ is the carrier for $x=(x_1,\ldots,x_N)\in X$ if $C_i = \carr_i(x_i)$ for $i=1,\ldots,N$ (or equivalently, if $C=\carr(x)$).

Let $\gamma_i := |C_i|$
and assume that the strategy set $Y_i$, is reordered so that $C_i = \{y_i^1,\ldots,y_i^{\gamma_i}\}$. Under this ordering, the first $\gamma_i-1$ components of any strategy $x_i$ with $\carr_i(x_i) = C_i$ are free (not constrained to zero by $C_i$) and the remaining components of $x_i$ are constrained to zero. That is $(x_i^k)_{k=1}^{\gamma_i-1}$ is free under $C_i$ and $(x_i^k)_{k=\gamma_i}^{K_i}=0$.
The set of strategies $\{x\in X:~\carr(x) = C\}$ is precisely the interior of the face of $X$ given by
\begin{equation}\label{def_C_face}
\Omega_C :=\{x\in X:~ x_i^k = 0,~ k=\gamma_i,\ldots,K_i-1,~i=1,\ldots,N \}.
\end{equation}

Let $x^*$ be an equilibrium with carrier $C$. We say that $x^*$ is \emph{first-order degenerate} if there exists a pair $(i,k)$, $i=1,\ldots,N$, $k=\gamma_i,\ldots,K_i-1$ such that $\frac{\partial U(x^*)}{\partial x_i^k}=0$, and we say $x^*$ is \emph{first-order non-degenerate} otherwise.

\begin{remark} \label{remark_first_order_equivalent_condition}
We note that using the multi-linearity of $U$, it is straightforward to verify that an equilibrium is first order non-degenerate if and only if it is quasi-strong, as introduced by Harsanyi \cite{harsanyi1973oddness} (see also \cite{van1991stability}). In particular, an equilibrium $x^*$ is first-order degenerate if and only if $\carr_i(x_i^*) \subsetneq \BR_i(x_{-i}^*)$ for some $i=1,\ldots,N$.  We prefer to use the term first order non-degenerate since it emphasizes that we are concerned with the gradient of the potential function and it keeps nomenclature consistent with the notion of second-order non-degeneracy, introduced next.
\end{remark}


\subsection{Second-Order Degeneracy} \label{sec_second_order_degeneracy}
Let $C$ be some carrier set. Let $\tilde{N}:=|\{i=1,\ldots,N:~\gamma_i \geq 2\}|$, and assume that the player set is ordered so that $\gamma_i\geq 2$ for $i=1,\ldots,\tilde{N}$. Under this ordering, for strategies with $\carr(x) = C$, the first $\tilde{N}$ players use mixed strategies and the remaining players use pure strategies.
Assume that $\tilde{N}\geq 1$ so that any $x$ with carrier $C$ is a mixed (not pure) strategy.

Let the Hessian of $U$ taken with respect to $C$ be given by
\begin{equation} \label{def_mixed_hessian}
\tilde{\vH}(x) :=\left( \frac{\partial^2 U(x)}{\partial x_i^k \partial x_j^\ell} \right)_{\substack{i,j=1,\ldots,\tilde{N},\\ k=1,\ldots,\gamma_i-1,\\ \ell=1,\ldots,\gamma_j-1}}.
\end{equation}
Note that this definition of the Hessian restricts attention to the components of $x$ that are free under $C$

We say an equilibrium $x^* \in X$ is \emph{second-order degenerate} if the Hessian
$\tilde{\vH}(x^*)$ taken with respect to $\carr(x^*)$
is singular, and we say $x^*$ is \emph{second-order non-degenerate} otherwise.

\begin{remark}
Note that both forms of degeneracy are concerned with the interaction of the potential function and the ``face'' of the strategy space containing the equilibrium $x^*$. If $x^*$ touches one or more constraints, then first-order non-degeneracy ensures that the gradient of the potential function is nonzero normal to the face $\Omega_{\carr(x^*)}$, defined in \eqref{def_C_face}. Second-order non-degeneracy ensures that, restricting $U$ to the face $\Omega_{\carr(x^*)}$, the Hessian of $U \big\vert_{\Omega_{\carr(x^*)}}$ is non-singular.
If $x^*$ is contained within the interior of $X$, then the first-order condition becomes moot and the second-order condition reduces to the standard definition of a non-degenerate critical point.
\end{remark}
\begin{remark}
Note that if an equilibrium $x^*$ is (first or second-order) degenerate with respect to some potential function $U$ for the game $\Gamma$, then it is likewise degenerate for every other admissible potential function for $\Gamma$. This justifies our usage of an arbitrary potential function $U$ associated with $\Gamma$ in the definitions of first and second order degeneracy.
\end{remark}

Throughout the paper we will study regular potential games.  The following lemma from \cite{swenson2017regular} shows that, in any regular potential game, all equilibria are first and second-order non-degenerate.
\begin{lemma} [\cite{swenson2017regular}, Lemma 12] \label{lemma_non_degen_to_regular}
Let $\Gamma$ be a potential game.
An equilibrium $x^*$ is regular if and only if it is both first and second-order non-degenerate.
\end{lemma}

\section{Potential Production Inequalities} \label{sec_inequalities}
In this section we prove two key inequalities (\eqref{equation_pot_inequality1_1} and \eqref{equation_pot_inequality2_1}) that are the backbone of our proof of Theorem \ref{thrm_main_result}.

We note that in proving Theorem \ref{thrm_main_result} there is a fundamental dichotomy between studying completely mixed equilibria and incompletely mixed equilibria.
Completely mixed equilibria lie in the interior of the strategy space. At these points the gradient of the potential function is zero and the Hessian is non-singular; local analysis of the dynamics is relatively easy. On the other hand, incompletely mixed equilibria necessarily lie on the boundary of $X$ and the potential function may have a nonzero gradient at these points.\footnote{We note that in games that are first-order non-degenerate, the gradient is always non-zero at incompletely mixed equilibria.} Analysis of the dynamics around these points is fundamentally more delicate.

In order to handle incompletely mixed equilibria we construct a nonlinear projection whose range is a lower dimensional game in which the image of the equilibrium under consideration is completely mixed. This allows us to handle both types of mixed equilibria in a unified manner.

\subsection{Projection to a Lower-Dimensional Game} \label{sec_diff_ineq_prelims}
Let $x^*$ be a mixed equilibrium.\footnote{We note that $x^*$ is assumed to be fixed throughout the section and many of the subsequently defined terms are implicitly dependent on $x^*$.} Let $C_i = \carr_i(x^*_i)$, where $x_i^*$ is the player-$i$ component of $x^*$, let $C=C_1\cup\cdots\cup C_N =\carr(x^*)$, and assume that $Y_i$ is ordered so that $\{y_i^1,\ldots,y_i^{\gamma_i}\} = C_i$.
Let $\gamma_i = |C_i|$, let $\tilde{N} :=\left|\big\{i\in\{1,\ldots,N\}:~\gamma_i \geq 2\big\}\right|$, and assume that the player set is ordered so that $\gamma_i \geq 2$ for $i=1,\ldots,\tilde{N}$.
Since $x^*$ is assumed to be a mixed-strategy equilibrium, we have $\tilde{N} \geq 1$.

Given an $x\in X$, we will frequently use the decomposition $x=(x_p,x_m)$, where $x_m := (x_i^k)_{i=1,\ldots,\tilde{N},~k=1,\ldots,\gamma_i-1}$ and $x_p$ contains the remaining components of $x$.\footnote{The subscript in $x_m$ is suggestive of ``mixed-strategy components'' and the subscript in $x_p$ is suggestive of ``pure-strategy components''. Furthermore, it is convenient to note that under the assumed ordering $\{y_i^1,\ldots,y_i^{\gamma_i}\} = C_i$ we have $x_p^*=0$; i.e., the pure strategy component at the equilibrium is equal to the null vector.} Let
$\gamma := \sum_{i=1}^N (\gamma_i-1).$
Recalling that $\kappa$ is the dimension of $X$ (see \eqref{def_kappa}), note that for $x\in X$ we have $x\in \R^\kappa$, $x_m\in\R^{\gamma}$, and $x_p\in \R^{\kappa-\gamma}$.

The set of joint pure strategies $Y$ may be expressed as an ordered set $Y=\{y^1,\ldots,y^K\}$ where each element $y^\tau\in Y$, $\tau\in\{1,\ldots,K\}$ is an $N$-tuple of strategies.
For each pure strategy $y^{\tau}\in Y$, $\tau=1,\ldots,K$, let $u^{\tau}$ denote the pure-strategy potential associated with playing $y^\tau$; that is, $u^\tau := u(y^\tau)$, where $u$ is the pure form of the potential function defined in Section \ref{sec_prelims}. A vector of \emph{potential coefficients} $u=(u^\tau)_{\tau=1}^K$ is an element of $\R^K$.

Given a vector of potential coefficients $u\in \R^K$ and a strategy $x\in X$,
let\footnote{We note that the functions $F_i^k$ and $F$ defined here are identical to those defined in (12) and (13) of \cite{swenson2017regular}, and used extensively throughout \cite{swenson2017regular}.}
\begin{equation} \label{def_F}
F_i^k(x,u):= \frac{\partial U(x)}{\partial x_i^k},
\end{equation}
for $i=1,\ldots,\tilde{N},~k=1,\ldots,\gamma_i-1$, and let
%
%
%
$$F(x,u) := \left( F_i^k(x,u)\right)_{\substack{i=1,\ldots,\tilde{N}\\ k=1,\ldots,\gamma_i-1}} =
\left( \frac{\partial U(x)}{\partial x_i^k}\right)_{\substack{i=1,\ldots,\tilde{N}\\ k=1,\ldots,\gamma_i-1}}.$$
Differentiating \eqref{eq_potential_expanded_form2} we see that at the equilibrium $x^*$ we have $\frac{\partial U(x^*)}{\partial x_i^k} = 0$ for $i=1,\ldots,\tilde{N}$, $k=1,\ldots,\gamma_i-1$ (see Lemma \ref{lemma_carrier_vs_gradient} in appendix), or equivalently,
$$
F(x^*,u) = F(x^*_p,x^*_m,u) = 0.
$$

By Definition \ref{def_equilibria_nomenclature}, the (mixed) equilibrium $x^*$ is completely mixed if $\gamma=\kappa$, and is incompletely mixed otherwise.
Suppose $\gamma<\kappa$ so that $x^*$ is incompletely mixed. Let $\vJ(x) := D_{x_m}F(x_p,x_m,u)$ and note that by definition we have $\vJ(x^*) = \tilde{\vH}(x^*)$.

Since $\Gamma$ is assumed to be a non-degenerate game, $\vJ(x^*)$ is invertible. By the implicit function theorem, there exists a function $g:\calD(g)\rightarrow \R^{\gamma}$ such that $F(x_p,g(x_p),u)$ $ = 0$ for all $x_p$ in a neighborhood of $x^*_p$, where  $\calD(g)\subset \R^{\kappa-\gamma}$ denotes the domain of $g$, $x_p^*\in\calD(g)$, and $\calD(g)$ is open.

The graph of $g$ is given by
$$\mbox{Graph}(g):=\{x\in X :~ x=(x_p,x_m),~ x_p \in \mathcal{D}(g),~ x_m = g(x_p)\}.$$
Note that $\Gr(g)$ is a smooth manifold with Hausdorff dimension $(\kappa-\gamma)$ \cite{EvansGariepy}.
An intuitive interpretation of $\Gr(g)$ is given in Remark \ref{remark_g}.

If $\Gamma$ is a non-degenerate potential game then, using the multilinarity of $U$, we see that $\gamma \geq 2$ (see Lemma \ref{lemma_two_players_mixing} in appendix). This implies that
\begin{equation}\label{eq_graph_size}
\Gr(g) \mbox{ has Hausdorff dimension at most }  (\kappa-2).
\end{equation}

Let $\Omega:=\Omega_{C}$, where $\Omega_{C}$ is defined in \eqref{def_C_face},
denote the face of $X$ containing $x^*$.
Define the mapping $\tilde \calP:\calD(\tilde{\calP}) \rightarrow \Omega$, with domain $\calD(\tilde{\calP}) := \{x=(x_p,x_m)\in X:~ x_p\in\calD(g)\}$, as follows. If $x^*$ is completely mixed then let $\tilde \calP(x) := x$ be the identity. Otherwise, let
\begin{equation} \label{def_calPtilde_map}
\tilde \calP(x) := x^*+\left(x - (x_p,g(x_p)) \right).
\end{equation}
Let $\tilde \calP_i^k(x)$ be the $(i,k)$-th coordinate map of $\tilde \calP$, so that $\tilde \calP = (\tilde \calP_i^k)_{\substack{i=1,\ldots,N\\ k=1,\ldots,K_i-1}}$.
Following the definitions, it is simple to verify that for $x\in \mathcal{D}(\tilde \calP)$ we have $\tilde \calP_i^k(x) = 0$ for all $(i,k)$ with $k\geq \gamma_i$, and hence $\tilde \calP$ indeed maps into $\Omega$.

Let $\tilde X_i := \{\tilde x_i\in \R^{\gamma_i-1}:~ \tilde x_i^k \geq 0, k=1,\ldots,\gamma_i-1,~\sum_{k=1}^{\gamma_i-1} \tilde x_i^k \leq 1 \}$, $i=1,\ldots,\tilde{N}$, and let $\tilde X := \tilde X_1\times\cdots\times \tilde X_{\tilde{N}}$.
Let $\calP:\calD(\calP)\rightarrow \tilde X$ with domain $\calD(\calP) = \calD(\tilde{\calP}) \subset X$ be given by
\begin{equation} \label{def_calP_map}
\calP := (\tilde \calP_i^k)_{i=1,\ldots,\tilde{N},~k=1,\ldots,\gamma_i-1}.
\end{equation}
Note that $\calP$ contains the components of $\tilde \calP$ not constrained to zero. As we will see in the following section, $\calP$ may be interpreted as a projection into a lower dimensional game in which $\calP(x^*)$ is a completely mixed equilibrium.


\begin{example} \label{example2}
Consider the two player game with payoff matrix.
\begin{center}
\vspace{.2cm}

\renewcommand{\arraystretch}{1.2}
\begin{tabular} {r|c|c|c|}
\multicolumn{4}{c}{\hspace{1.5em} $A$\hspace{1.4em} $B$ \hspace{1.4em} $C$}\\
\cline{2-4}
$A$ & $1,~1$ & $0,~0$ & $0,~0$ \\
\cline{2-4}
$B$ & $0,~0$ & $2,~2$ & $0,~0$\\
\cline{2-4}
\end{tabular}
\renewcommand{\arraystretch}{1.0}
\vspace{.4cm}

\end{center}

\begin{figure}[h]
    \centering
    \begin{subfigure}{.45\textwidth}
        \includegraphics[height=.85\textwidth]{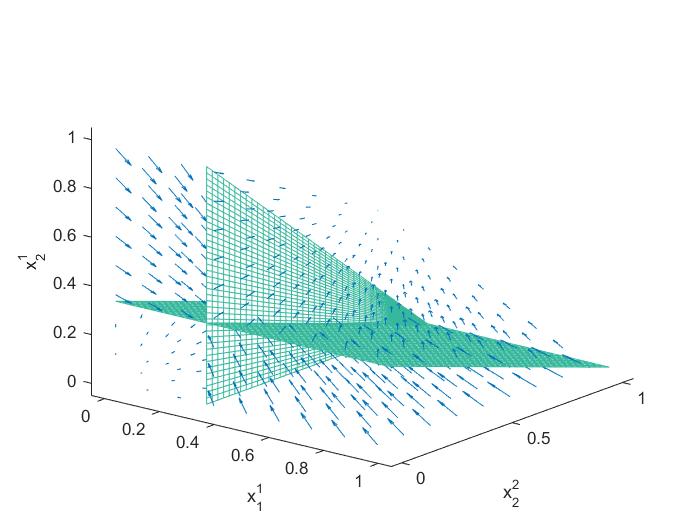}
        \caption{BRD vector field from Example \ref{example2}.}
        \label{fig:gExample1}
    \end{subfigure}
    \hspace{.8cm}
    \begin{subfigure}{.45\textwidth}
        \includegraphics[height=.85\textwidth]{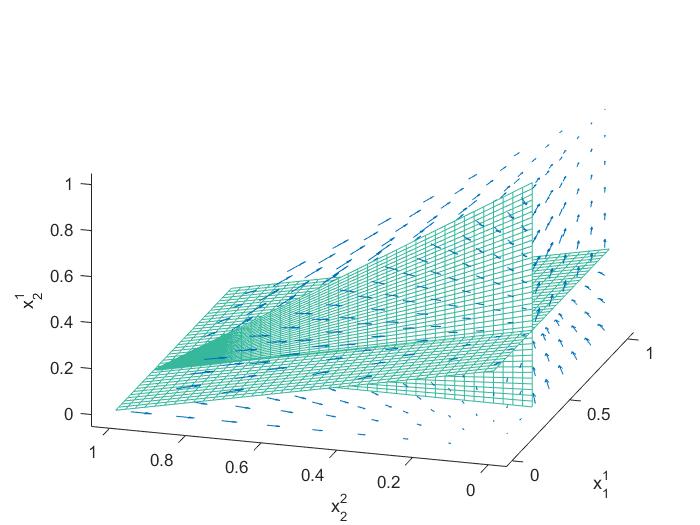}
        \caption{Alternate view of BRD vector field from Example \ref{example2}.}
        \label{fig:gExample2}
    \end{subfigure}
        \begin{subfigure}{.45\textwidth}
        \includegraphics[height=.85\textwidth]{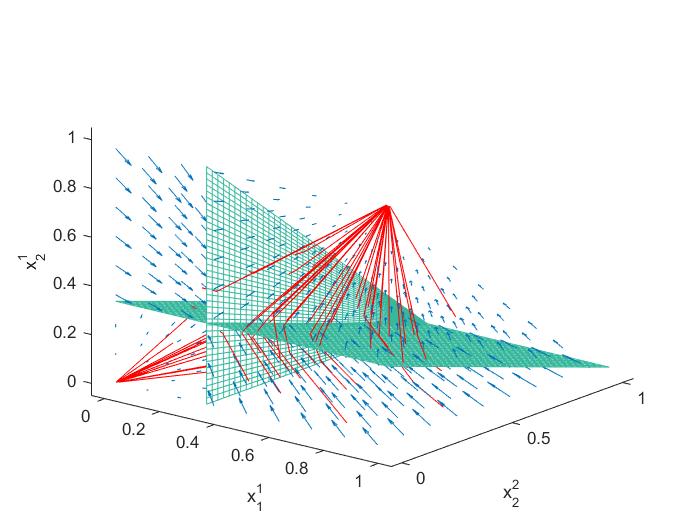}
        \caption{Plot of BR solution curves from Example \ref{example2}.}
        \label{fig:gExample3}
    \end{subfigure}
    \hspace{.8cm}
    \begin{subfigure}{.45\textwidth}
        \includegraphics[height=.85\textwidth]{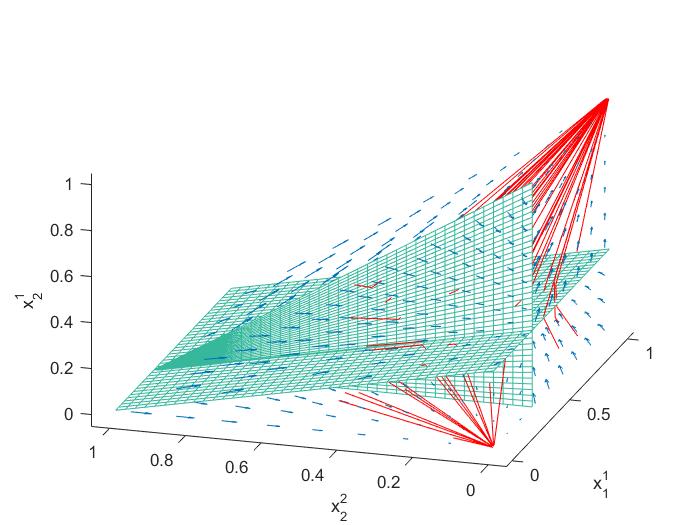}
        \caption{Alternate view of plot of BR solution curves from Example \ref{example2}.}
        \label{fig:gExample4}
    \end{subfigure}
    \caption{}
    \vspace{-1.2em}
\end{figure}

We will refer to the row player as player 1 and the column player as player 2.
Note that this is equivalent to the game in Example \ref{example_FP1} where the player 2 has been given an additional action yielding a uniform payoff of zero for both players.
Following the conventions of Section \ref{sec_notation}, let $x_1^1$ denote the probability of player 1 playing action $B$, and let $x_2^1$ and $x_2^2$ denote the probabilities of player $2$ playing $B$ and $C$, respectively. The mixed strategy space $X$ for this game is a triangular cylinder---a plot of the BRD vector field for this game is shown from two different perspectives in Figures \ref{fig:gExample1}--\ref{fig:gExample2}, where the blue arrows give the direction of the vector field, and the green surfaces represent regions where some player is indifferent between actions (i.e., ``indifference surfaces''). Note that the vector field jumps along these surfaces.

Let $\Omega$ denote the face of $X$ corresponding to $x_2^2 = 0$ (i.e., the face of $X$ when we restrict player 2 to place weight 0 on action $C$). Note that the vector field within $\Omega$ is identical to the  familiar $2\times 2$ vector field from Example \ref{example_FP1}.


Let $x^*$ be the equilibrium $((x_1^1),(x_2^1,x_2^2)) = ((1/3),(1/3,0))$. The graph of the associated function $g$ coincides with intersection of the indifference surfaces emanating from $x^*$ out of $\Omega$. (In general, the graph of $g$ will always correspond to the intersection of indifference surfaces connecting to the equilibrium and extending out of $\Omega$.) The projection $\calP$ projects strategies from $X$ into the face $\Omega$ as defined in \eqref{def_calPtilde_map} and \eqref{def_calP_map}.

Fifty solution curves of the BR dynamics in this game with random initial conditions are plotted in red in Figures \ref{fig:gExample3}--\ref{fig:gExample4}. Note that the solution curves converge to pure equilibria.

\end{example}

\subsection{Inequalities}
Let $\tilde{U}:\tilde X \rightarrow \R$ be given by
$$
\tilde{U}(\tilde x) := U(x_p^*, \tilde x),
$$
where $x^* = (x_p^*,x_m^*)$ is the mixed equilibrium fixed in the beginning of the section.
Let $\tilde{\Gamma}$ be a potential game with player set $\{1,\ldots,\tilde{N}\}$, mixed-strategy space $\tilde X_i$, $i=1,\ldots,\tilde{N}$, and potential function $\tilde{U}$. By construction, $\calP(x^*)$ is a completely mixed equilibrium of $\tilde{\Gamma}$. Moreover, by the definition of a non-degenerate equilibrium, the Hessian of $\tilde{U}$ is invertible at $\calP(x^*)$.

We are interested in studying the projection $\calP(\vx(t))$ of a BR process into the lower dimensional game $\tilde{\Gamma}$.\footnote{In the lower dimensional game $\tilde{\Gamma}$, the dynamics of the projected process are not precisely BR dynamics. However, the behave \emph{nearly} like BR dynamics, which is what allows us to establish these inequalities.}
We wish to show that the following two inequalities hold:\\
(i) For $x$ in a neighborhood of $x^*$
\begin{equation}\label{equation_pot_inequality1_1}
\big\vert \tilde{U}(\calP(x^*)) - \tilde{U}(\calP(x)) \big\vert \leq c_1 d^2(\calP(x),\calP(x^*)),
\end{equation}
for some constant $c_1>0$.\\
(ii) Suppose $(\vx(t))_{t\geq 0}$ is a BR process. For $\vx(t)$ residing in a neighborhood of $x^*$
\begin{equation}\label{equation_pot_inequality2_1}
\frac{d}{dt} \tilde{U}(\calP(\vx(t))) \geq c_2 d(\calP(\vx(t)),\calP(x^*)),
\end{equation}
for some constant $c_2>0$.\footnote{We note that when we write these inequalities, we mean they are satisfied in an integrated sense (e.g., as used in \eqref{eq_inf_time_pf_eq1}--\eqref{eq_inf_time_pf_eq2}). In this section, we treat all of these as pointwise inequalities. A rigorous argument could be constructed using the chain rule in Sobolev spaces (see, for example, \cite{MR2293955}).}

The first inequality follows from Taylor's theorem and the fact that $\nabla \tilde{U}(\calP(x^*)) = 0$.
The following  two sections are devoted to proving \eqref{equation_pot_inequality2_1}.
In order to build intuition and

In Section \ref{sec_completely_mixed_case} we consider the simple case in which $x^*$ is a completely mixed (interior) equilibrium.
Subsequently, in Section \ref{sec_incompletely_mixed_case} we consider the more complicated case in which $x^*$ is an incompletely mixed equilibrium. The basic idea of the proof of \eqref{equation_pot_inequality2_1} in the completely and incompletely mixed cases is the same. However, care must be taken to appropriately handle problems with possible first-order degeneracies occurring at incompletely mixed equilibria.
The reader may wish to skip Section \ref{sec_incompletely_mixed_case} on a first read-through.

\subsection{Proving the Differential Inequality: The Completely Mixed Case} \label{sec_completely_mixed_case}
We begin with Lemma \ref{lemma_IR1} which shows---roughly speaking---that within the interior of the action space, the BR-dynamics vector field approximates the gradient field of the potential function.

The following definitions are useful in the lemma.
For $B\subseteq X$, let $P_{X_i}(B) := \{x_i\in X_i:~ (x_i,x_{-i}) \in B \mbox{ for some } x_{-i} \in X_{-i}\}$ be the projection of $B$ onto $X_i$.
Given an $x_i\in X_i$, let
$$
d(x_i,\partial X_i) := \min\{x_i^1,\ldots,x_i^{K_i-1},1-\sum_{k=1}^{K_i-1} x_i^k\}
$$
denote the distance from $x_i$ to the boundary of $X_i$. Let
$$
d(P_{X_i}(B),\partial X_i) := \inf_{x_i\in P_{X_i}(B)} d(x_i,\partial X_i)
$$
denote the distance between the set $P_{X_i}(B)$ and the boundary of $X_i$.

Since we will eventually be interested in studying a lower-dimensional game derived from $\Gamma$, in the lemma we consider an alternative game $\hat{\Gamma}$ of arbitrary size.
\begin{lemma}\label{lemma_IR1}
Let $\hat{\Gamma}$ be a potential game with player set $\{1,\ldots,\hat{N}\}$, action sets $\hat{Y}_i$, $i=1,\ldots,\hat{N}$, with cardinality $\hat{K}_i:=|\hat{Y}_i|$, and potential function $\hat{U}$. Let $\hat{X}=\hat{X}_1\times\cdots\times \hat{X}_{\hat{N}}$ denote the mixed strategy space.

Let $B\subset \hat{X}$ and fix $i\in\{1,\ldots,\hat{N}\}$.
Then for all $x\in B$ there holds
\begin{equation} \label{IR1_eq0}
z_i\cdot \nabla_{x_i} \hat{U}(x)\geq c \|\nabla_{x_i} \hat{U}(x)\|_1, \quad\quad \forall~ z_i\in \BR_i(x_{-i}) - x_i
\end{equation}
where the constant $c$ is given by $c = d(P_{X_i}(B),\partial \hat{X}_i)$.
\end{lemma}
\begin{proof}
Let $x\in B$. If $\|\nabla_{x_i}\hat{U}(x)\|_1=0$, then $\nabla_{x_i} \hat{U}(x)=0$, and the inequality is trivially satisfied. Suppose from now on that $\|\nabla_{x_i}\hat{U}(x)\|_1 >0$.

Without loss of generality, assume that $Y_i$ is ordered so that
\begin{equation}\label{IR1_eq4}
\actionione \in \BR_i(x_{-i}).
\end{equation}
Differentiating \eqref{eq_potential_expanded_form2} we find that\footnote{Note that the domain of (expected) potential function $\hat{U}$ may be trivially extended to an open neighborhood around $\hat{X}$ (see Section \ref{sec_prelims}). Using this extension we see that the derivative is well defined for $x$ lying on the boundary of $\hat{X}$.}
\begin{equation}\label{IR1_eq2}
\frac{\partial \hat{U}(x)}{\partial x_i^{k}} = \hat{U}(\actionikone,x_{-i}) - \hat{U}(\actionione,x_{-i}).
\end{equation}
Together with \eqref{IR1_eq4}, this implies that for $k=1,\ldots,\hat{K}_i-1$ we have
\begin{equation} \label{IR1_eq1}
\actionikone \in \BR_i(x_{-i}) \iff \frac{\partial \hat{U}(x)}{\partial x_i^k}=0.
\end{equation}
Using the multlinearity of $\hat{U}$ we see that if $\xi_i\in \BR_i(x_{-i})$ and $\xi_i^k > 0$ then $\actionikone \in \BR_i(x_{-i})$. But, by \eqref{IR1_eq1} this implies that if $\xi_i\in \BR_i(x_{-i})$ and $\xi_i^k > 0$ then $\frac{\partial \hat{U}(x)}{\partial x_i^k} = 0$.
Noting that any $\xi_i\in \BR_i(x_{-i})$ is necessarily coordinatewise nonnegative, this gives
\begin{align} \label{IR1_eq3}
(\xi_i - x_i)\cdot \nabla_{x_i} \hat{U}(x) = \underbrace{\sum_{k=1}^{\hat{K}_i-1} \xi_i^k \frac{\partial \hat{U}(x)}{\partial x_i^{k}}}_{=0} - \sum_{k=1}^{\hat{K}_i-1} x_i^k \frac{\partial \hat{U}(x)}{\partial x_i^{k}}, \quad\quad \xi_i\in \BR_i(x_{-i})
\end{align}
Since we assume $x\in B$, we have $x_i^k \geq d(P_{\hat{X}_i}(B),\partial \hat{X}_i)$, for all $k=1,\ldots,\hat{K}_i-1$.
Since we assume $y_i^1\in \BR_i(x_{-i})$, from \eqref{IR1_eq2} we get that $\frac{\partial \hat{U}(x)}{\partial x_i^{k}}\leq 0$ for all $k=1,\ldots,\hat{K}_i-1$. Substituting into \eqref{IR1_eq3}, this gives
\begin{equation}
(\xi_i - x_i)\cdot \nabla_{x_i} \hat{U}(x) \geq d(P_{\hat{X}_i}(B),\partial \hat{X}_i) \sum_{k=1}^{\hat{K}_i-1} \left(-\frac{\partial \hat{U}(x)}{\partial x_i^{k}}\right), \quad\quad \xi_i \in \BR_i(x_{-i}).
\end{equation}
But since $\frac{\partial \hat{U}(x)}{\partial x_i^{k}}\leq 0$ for all $k$ we have $\sum_{k=1}^{\hat{K}_i-1} \left(-\frac{\partial \hat{U}(x)}{\partial x_i^{k}}\right) = \|\nabla_{x_i}\hat{U}(x)\|_{1}$, and hence
$$
(\xi_i - x_i)\cdot \nabla_{x_i} \hat{U}(x) \geq d(P_{\hat{X}_i}(B),\partial \hat{X}_i) \|\nabla_{x_i}\hat{U}(x)\|_1, \quad\quad \xi_i \in \BR_i(x_{-i}),
$$
which is the desired result.
\end{proof}

\begin{remark} \label{remark_norm_equivalence}
Since the space $X_i$ in Lemma \ref{lemma_IR1} is finite dimensional, given any norm $\|\cdot\|$, there exists a constant $\tilde c>0$ such that
$$
z_i\cdot \nabla_{x_i} U(x)\geq c \|\nabla_{x_i} U(x)\|, \quad\quad ~ \forall ~z_i\in \BR_i(x_{-i})-x_i
$$
with $c = \tilde c d(P_{X_i}(B),\partial X_i)$.
\end{remark}

The following lemma proves \eqref{equation_pot_inequality2_1} for the case in which $x^*$ is a completely mixed equilibrium. Note that in this case the projection $\calP$ is given by the identity, so \eqref{equation_pot_inequality2_1} becomes
\begin{equation} \label{eq_diff_inequal2}
\ddt U(\vx(t)) \geq c_2 d(\vx(t),x^*).
\end{equation}
\begin{lemma} Suppose $x^*$ is a completely mixed equilibrium. Then  \eqref{eq_diff_inequal2} holds for $\vx(t)$ in a neighborhood of $x^*$.
\end{lemma}
\begin{proof}
Note that
$$
\ddt U(\vx(t)) = \nabla U(\vx(t))\cdot \dot \vx(t)\\
= \sum_{i=1}^N \nabla_i U(\vx(t)) \cdot z_i
$$
for some $z_i\in \BR_i(\vx_{-i}) - \vx_i$.
By Lemma \ref{lemma_IR1} this gives
$$
\ddt U(\vx(t)) \geq \sum_{i=1}^N c \|\nabla_{x_i} U(\vx)\|.
$$
By the equivalence of finite-dimensional norms, there exists a constant $c_1$ such that $\ddt U(\vx(t))\geq c_1\|\nabla U(\vx(t))\|$ for $\vx(t)$ in a neighborhood of $x^*$. Since $\Gamma$ is assumed to be regular (and hence second-order non-degenerate), $x^*$ is a non-degenerate critical point of $U$. By Lemma \ref{lemma_gradient_inequality_apx} (see appendix) there exists a constant $c_2$ such that $c_1\|\nabla U(x)\| \geq c_2 d(x,x^*)$ for $x$ in a neighborhood of $x^*$, and hence $\ddt U(\vx(t)) \geq c_2 d(\vx(t),x^*)$.
\end{proof}

\subsection{Proving the Differential Inequality: The Incompletely Mixed Case} \label{sec_incompletely_mixed_case}
In this section we prove \eqref{equation_pot_inequality2_1} for the case in which $x^*$ is incompletely mixed. The main idea of the proof is the same as the proof in the completely mixed case. However, care must be taken to ensure that $\vx(t)$ approaches the boundary of $X$ in an appropriate manner. Handling this case is the principal role of the first-order non-degeneracy condition.

For each $x=(x_p,x_m)\in X$ near to $x^*$, the following lemma allows us to define an additional lower dimensional game $\Gamma_{x_p}$ associated with $x_p$ in which the best-response set is closely related to the best-response set for the original game $\Gamma$.
The lemma is a straightforward consequence of the definition of the best response correspondence and the continuity of $\nabla U$.
\begin{lemma} \label{lemma_IR4}
For $x$ in a neighborhood of $x^*$, the best response set satisfies
$$
\BR_i(x_{-i}) \subseteq \BR_i(x^*_{-i}),\quad \quad \forall ~i=1,\ldots,\tilde{N}.
$$
\end{lemma}

Given any $x = (x_p,x_m) \in X$ we define $\tilde U_{x_p}:\tilde X\rightarrow \R$ and $\BRtilde_{x_p,i}:\tilde X_{-i}\rightrightarrows \tilde X_i$ as follows. For $\tilde x \in \tilde X$ let
\begin{equation} \label{eq_V_tilde}
\tilde U_{x_p}(\tilde x) :=  U(x_p, \tilde x),
\end{equation}
and for $\tilde x_{-i} \in \tilde X_{-i}$ let
\begin{equation} \label{eq_BR_tilde}
\BRtilde_{x_p,i}(\tilde x_{-i}) := \arg\max_{\tilde{x}_i \in \tilde X_i} \tilde U_{x_p}(\tilde{x}_i,\tilde x_{-i})
\end{equation}

Let $\Gamma_{x_p}$ be the potential game with player set $\{1,\ldots,\tilde{N}\}$, mixed strategy space $\tilde X$ and potential function $\tilde U_{x_p}$. Note that since $U$ is continuous and $\tilde X$ is compact, $\tilde U_{x_p}$ converges uniformly to $\tilde U_{x_p^*} =: \tilde{U}$ as $x_p \rightarrow x_p^*$. In this sense the game $\Gamma_{x_p}$ can be seen as converging to $\tilde{\Gamma}$ as $x_p \rightarrow x_p^*$.
\begin{remark}\label{remark_g}
The function $g$ defined in Section \ref{sec_diff_ineq_prelims} admits the following interpretation.
Suppose we fix some $x_p = (x_i^k)_{i=1,\ldots,N,~k=\gamma_i,\ldots,K_i-1}$.
Then $g(x_p)$ is a completely mixed Nash equilibrium of $\Gamma_{x_p}$. Moreover, if we let $x_p\rightarrow x_p^*$, then the corresponding equilibrium of the reduced game $\Gamma_{x_p}$ converges to $x^*$, i.e., $(x_p,g(x_p)) \rightarrow (x_p^*,g(x_p^*)) = x^*$, precisely along $\Gr(g)$. (See Example \ref{example2} for an illustration.)
\end{remark}
\begin{remark} \label{remark_first_order_degeneracy}
Suppose $x^*$ is a first-order non-degenerate equilibrium. Using the multilinearity of $U$ we see that for any $x\in X$ we have $\carr_i(x_i) \subseteq \BR_i(x_{-i})$. By Remark \ref{remark_first_order_equivalent_condition}, at $x^*$ we have $\carr_i(x_i^*) = \BR_i(x_{-i}^*)$.
Due to the ordering we assumed on $Y_i$, this implies that
$y_i^k \in \BR_i(x_{-i}^*) \iff 1\leq k\leq \gamma_i$. Moving to the $X$ domain, this means that if $\hat{x}_i \in \BR_i(x_{-i}^*)$, then $\hat{x}_i^k = 0$ for all $k=\gamma_i,\ldots,K_i-1$. By Lemma \ref{lemma_IR4}, this implies that for all $x$ in a neighborhood of $x^*$ and for $\hat{x}_i \in \BR_i(x_{-i})$ we have $(\hat{x}_i^k)_{k=\gamma_i}^{K_i-1} = 0$.
\end{remark}

The following lemma extends the result of Lemma \ref{lemma_IR1} so it applies in a useful way to the potential function $\tilde U$ under the projection $\calP$.
\begin{lemma} \label{lemma_IR5}
There exists a constant $c>0$ such that for all $x=(x_p,x_m)$ in a neighborhood of $x^*$ and all $\eta\in\R^{\gamma_i-1}$ with $\|\eta\|$ sufficiently small we have
$$
(z_i + \eta) \cdot \nabla_{x_i} \tilde{U}(\calP(x)) \geq c\|\nabla_{x_i} \tilde{U}(\calP( x))\|,
$$
for all $z_i \in \BRtilde_{i,x_p}([x_m]_{-i}) - [x_m]_i$, where $[x_m]_i := (x_i^k)_{k=1,\ldots,\gamma_i-1}$ refers to the player-$i$ component of $x_m$ and $[x_m]_{-i}$ contains the components of $x_m$ corresponding to the remaining players.
\end{lemma}
The proof of this lemma is relatively straightforward and omitted for brevity.

Finally, the following lemma shows that the differential inequality \eqref{equation_pot_inequality2_1} holds.
\begin{lemma}\label{lemma_IR3}
Let $\Gamma$ be a non-degenerate potential game with mixed equilibrium $x^*$, and let $(\vx(t))_{t\geq 0}$ be a BR process. Then the inequality \eqref{equation_pot_inequality2_1} holds for $\vx(t)$ in a neighborhood of $x^*$.
\end{lemma}
\begin{proof}
Let
$$
\vP(x) := \left( \frac{\partial \tilde \calP_i^k}{\partial x_j^\ell } \right)_{\substack{i=1,\ldots,\tilde{N}, k=1,\ldots,\gamma_i-1\\ j=1,\ldots,N,~ \ell=\gamma_i,\ldots,K_i-1}}
$$
where the partial derivatives are evaluated at $x$.
The Jacobian of $\tilde \calP$ evaluated at $x$ is given by
$$
\left( \frac{\partial \tilde \calP_i^k}{\partial x_j^\ell}\right)_{\substack{i,j=1,\ldots,N,\\ k,\ell=1,\ldots,K_i-1}} =
\begin{pmatrix}
I & ~ \vP(x)\\
0 & ~ 0
\end{pmatrix}.
$$
Using the chain rule we may express the time derivative of the potential along the path $\tilde \calP(\vx(t))$ as
$$
\frac{d}{dt} U(\tilde \calP(\vx(t))) = \nabla U(\tilde \calP(\vx(t)))
\begin{pmatrix}
I & ~ \vP(x)\\
0 & ~ 0
\end{pmatrix} \dot{x}
= \nabla_{x_m} U(\tilde \calP(\vx(t)))(\vI ~~\vP(x))\dot{x}.
$$
For $i=1,\ldots,\tilde{N}$, $k=1,\ldots,\gamma_i-1$ let $\eta_i^k(t) := \sum_{j=1}^{N}\sum_{\ell=\gamma_j}^{K_i-1} \frac{\partial \tilde \calP_i^k}{\partial x_j^\ell}\dot x_i^\ell$, let $\eta_i(t) := (\eta_i^k(t))_{k=1}^{\gamma_i-1}$, and let $\eta(t) = (\eta_i(t))_{i=1}^{\tilde{N}}$.
Multiplying out the right two terms above we get
\begin{equation} \label{IR3_eq1}
\frac{d}{dt}U(\tilde \calP(\vx(t))) = \nabla_{x_m} U(\tilde \calP(\vx(t))) \, (\dot{x}_m + \eta(t))
\end{equation}
By Lemma \ref{lemma_IR4} and Remark \ref{remark_first_order_degeneracy}, if we restrict $\vx(t)$ to a sufficiently small neighborhood of $x^*$ then for any $z_i = (z_i',z_i'') \in \BR_i(\vx_{-i}(t))$, $z_i'=(z_i^k)_{k=1}^{\gamma_i-1}$, $z_i''=(z_i^k)_{k=\gamma_i}^{K_i-1}$, we have $z_i' \in \BRtilde_{x_p,i}([x_m]_{-i})$ and $z_i'' = 0$. We note two important consequences of this:\\
(i) If we restrict $\vx(t)$ to a sufficiently small neighborhood of $x^*$ and note that \newline $U(\tilde \calP(\vx(t))) = \tilde{U}(\calP(\vx(t))$, then by \eqref{IR3_eq1} we have
\begin{align} \label{IR3_eq2}
\frac{d}{dt} \tilde{U}(\calP(\vx(t))) & =
\nabla \tilde{U}(\calP(\vx(t))) \cdot
\begin{pmatrix}
z_1(t) + \eta_1(t)\\
\vdots\\
z_{\tilde{N}}(t) + \eta_{\tilde{N}}(t)
\end{pmatrix}\\
\nonumber & = \sum_{i=1}^{\tilde{N}}\nabla_{x_i} \tilde{U}(\calP(\vx(t))) \cdot (z_i(t) + \eta_i(t)),
\end{align}
where $z_i(t) \in \BRtilde_{x_p(t),i}([x_m(t)]_{-i}) - [x_m(t)]_i$.\\
(ii) We may force $\max_{i=1,\ldots,\tilde{N}} \|\eta_i\|$ to be arbitrarily small by restricting $\vx(t)$ to a neighborhood of $x^*$.

Consequence (i) follows readily by using the definition of the BR dynamics \eqref{def_FP_autonomous}. To show consequence (ii), note that by \eqref{def_FP_autonomous} we have $\dot{\vx}_i^k = z_i^k - x_i^k$ for all $i=1,\ldots,N$, $k=1,\ldots,K_i$, for some $z_i \in \BR_i(x_{-i})$. But, for $x$ in a neighborhood of $x^*$ and $k\geq \gamma_i$, we have shown above that $z_i^k = 0$, and hence $\dot{x}_i^k = -x_i^k$.\footnote{We note that this particular step depends crucially on the assumption of first-order non-degeneracy (see Remark \ref{remark_first_order_degeneracy}).} Due the ordering we assumed for $Y_i$, we have $[x^*]_i^k = 0$ for any $(i,k)$ such that $k\geq \gamma_i$. Hence, $x_i^k \rightarrow 0$ as $x\rightarrow x^*$, for any $(i,k)$ such that $k\geq \gamma_i$.

Furthermore, there exists a $c>0$ such that $|\frac{\partial \tilde \calP_i^k(x)}{\partial x_j^\ell}| < c$, $i=1,\ldots,\tilde{N}$, $k=1,\ldots,\gamma_i-1$, $j=1,\ldots,N$, $\ell\geq \gamma_j$ uniformly for $x$ in a neighborhood of $x^*$ (see Lemma \ref{lemma_g_derivative_finite} in appendix).
By the definition of $\eta_i$, this implies that $\max_{i=1,\ldots,\tilde{N}} \|\eta_i\|$ may be made arbitrarily small by restricting $\vx(t)$ to a sufficiently small neighborhood of $x^*$.

Now, let $\vx(t)$ be restricted to a sufficiently small neighborhood of $x^*$ so that $\|\eta_i(t)\|$ is small enough to apply Lemma \ref{lemma_IR5} for each $i$. Applying Lemma \ref{lemma_IR5} to \eqref{IR3_eq2} we get
$
\frac{d}{dt} \tilde{U}(\calP(\vx(t))) \geq \sum_{i=1}^{\tilde{N}} c \|\nabla_{x_i} \tilde{U}(\calP(\vx(t)))\|
$
for $\vx(t)$ in a neighborhood of $x^*$. By the equivalence of finite-dimensional norms, there exists a constant $c_1$ such that $\frac{d}{dt} \tilde{U}(\calP(\vx(t))) \geq  c_1\|\nabla \tilde{U}(\calP(\vx(t)))\|$
for $\vx(t)$ in a neighborhood of $x^*$.

Since $\Gamma$ is assumed to be (second-order) non-degenerate, $\calP(x^*)$ is a non-degenerate critical point of $\tilde U$. By Lemma \ref{lemma_gradient_inequality_apx} (see appendix) there exists a constant $c_2$ such that
$c_1\|\nabla \tilde{U}(\tilde x)\| \geq c_2d(\tilde x,\calP(x^*))$ for all $\tilde x\in \tilde X$ in a neighborhood of $\calP(x^*)$.
Since $\calP$ is continuous we have
$
\frac{d}{dt} \tilde{U}(\calP(\vx(t))) \geq c_2 d(\calP(\vx(t)),\calP(x^*)),
$
for $\vx(t)$ in a neighborhood of $x^*$.
\end{proof}

\section{Proof of Main Result} \label{sec_proof_main_result}
We will assume throughout this section that $\Gamma$ is a regular potential game. By Theorem 1 of \cite{swenson2017regular}, the ensuing results hold for almost all potential games.

For each mixed equilibrium $x^*$, let the set $\Lambda(x^*)\subset X$ be defined as
$$
\Lambda(x^*) :=
\begin{cases}
\{x^*\} & \mbox{ if $x^*$ is completely mixed},\\
\Gr(g) & \mbox{ otherwise},
\end{cases}
$$
where $g$ is defined with respect to $x^*$ as in Section \ref{sec_diff_ineq_prelims}.

In this section we will prove Theorem \ref{thrm_main_result} in two steps. First, we will show that for each mixed equilibrium $x^*$, the set $\Lambda(x^*)$ can only be reached in finite time from an $\calL^\kappa$-null set of initial conditions (see Proposition \ref{prop_finite_time_convergence}), where $\kappa$, defined in \eqref{def_kappa}, is the dimension of $X$.
Second, we will show that if a BR process converges to the set $\Lambda(x^*)$,
then it must do so in finite time (see Proposition \ref{prop_infinite_time_convergence}). Since $x^*\in \Lambda(x^*)$, Propositions \ref{prop_finite_time_convergence} and \ref{prop_infinite_time_convergence} together show that for any mixed equilibrium $x^*$, the set of initial conditions from which BR dynamics converge to $x^*$ has $\calL^\kappa$-measure zero.

By Theorem 2 of \cite{swenson2017regular} we see that in regular potential games, the set of NE is finite. Hence, Propositions \ref{prop_finite_time_convergence} and \ref{prop_infinite_time_convergence} imply that BR dynamics can only converge to \emph{set} of mixed strategy equilibria from a $\calL^\kappa$-null set of initial conditions.
Since a BR process must converge to the set of NE in a potential game (\cite{benaim2005stochastic}, Theorem 5.5), this implies that Theorem \ref{thrm_main_result} holds.

\subsection{Finite-Time Convergence} \label{sec_finite_time_convergence}
The goal of this subsection is to prove the following proposition.
\begin{proposition} \label{prop_finite_time_convergence}
Let $\Gamma$ be a non-degenerate game and let $x^*$ be a mixed-strategy NE of $\Gamma$. The set $\Lambda(x^*)$ can only be reached by a BR process in finite time from a set of initial conditions with $\calL^\kappa$-measure zero. That is,
\begin{align*}
\calL^\kappa(\{x_0\in X: \vx(0) = x_0,~ & \vx(t) \mbox{ is a BR process},\\
& \vx(t)\in \Lambda(x^*) \mbox{ for some } t\in[0,\infty)\})=0.
\end{align*}
\end{proposition}

We will take the following approach in proving the proposition. First, we will establish that solutions of \eqref{def_FP_autonomous} are unique (over a finite-time horizon) almost everywhere in $X$ (see Lemma \ref{lemma_unique_flow}). We will then show that---in an appropriate measure-theoretic sense---the BR-dynamics vector field has bounded divergence (see Lemma \ref{lemma_FP_divergence}). The practical implication of this result will be that BR dynamics cannot compress a set of positive measure into a set of zero measure in finite time. Since $\Lambda(x^*)$ is a low-dimensional set (see below), we will see that this implies that the set from which $\Lambda(x^*)$ can be reached in finite time cannot have positive measure, which will prove the proposition.

Before proving the proposition we present some definitions and preliminary results.
Let
\begin{equation} \label{def_indiff_surface}
\calI_{i,k,\ell}:=\{(x_i,x_{-i})\in X:~ U(y_i^k,x_{-i}) = U(y_i^{\ell},x_{-i})\},
\end{equation}
for $i=1,\ldots,N$, $k,\ell =1,\ldots,K_i$, $\ell\not= k$, be the set in which player $i$ is indifferent between his $k$-th and $\ell$-th actions.


If the game $\Gamma$ is non-degenerate, then each $\calI_{i,k,\ell}$ is the union of smooth surfaces with Hausdorff dimension at most $(k-1)$ (see Lemma \ref{lemma_thin_indifference_surface} in appendix).
In particular, for each $x\in \calI_{i,k,\ell}$ there exists a vector $\nu\in\R^{\kappa}$ that is normal to $\calI_{i,k,\ell}$ at $x$. We refer to the set $\calI_{i,k,\ell}$ as an \emph{indifference surface} of player $i$.

We define the set $\tilde Q\subseteq X$ as follows. Let $\tilde Q$ contain the set of points where two or more indifference surfaces intersect and their normal vectors do not coincide. Furthermore, if an indifference surface $\calI$ has a component $\hat{\calI}\subseteq \calI$ with Hausdorff dimension less than $\kappa-1$, then we put any points where $\hat{\calI}$ intersects with another decision surface into $\tilde{Q}$. Since each indifference surface is smooth with dimension at most $\kappa-1$, $\tilde Q$ has Hausdorff dimension at most $\kappa-2$. Let
$$
Q:= \tilde Q\cup \Lambda(x^*).
$$
As shown in Section \ref{sec_diff_ineq_prelims}, if $x^*$ is non-degenerate, then the set $\Gr(g)$ (and hence $\Lambda(x^*)$) has Hausdorff dimension at most $\kappa-2$.
Thus $Q$ has Hausdorff dimension at most $\kappa-2$.\footnote{Proposition \ref{prop_finite_time_convergence} can easily be generalized to say that any set $A\subset X$ such that $\cl A$ has Hausdorff dimension at most $\kappa-2$, can only be reached in finite time from a set of $\calL^\kappa$-measure zero by substituting $A$ for $\Lambda(x^*)$ throughout the section.}

The BR-dynamics vector field (see \eqref{def_FP_autonomous}) is given by the map $\FP:X\rightrightarrows X$, where
\begin{equation} \label{def_FP_map}
\FP(x) := \BR(x) - x.
\end{equation}

Let
\begin{align}\label{def_calZ}
 \calZ := \big\{x\in X\backslash Q:~x\in \calI_{i,k,\ell}~ \mbox{ for some } i,k,\ell \mbox{ with normal } \nu  & \mbox{ at } x, \\
 \mbox{ and } \nonumber \nu\cdot z = 0 \mbox{ for some }  z\in & ~\FP(x)\big\}.
\end{align}

Since each $\calI_{i,k,\ell}$ has Hausdorff dimension at most $\kappa-1$, $\calZ$ has Hausdorff dimension at most $\kappa-1$. We define the \emph{relative boundary} of $\calZ$, denoted here as $\partial \calZ$ as follows.
If $\calZ$ has Hausdorff dimension $\kappa-2$ or less, then let $\partial \calZ := \calZ$.
If $\calZ$ has Hausdorff dimension $\kappa-1$ then it may be expressed as the union of a finite number of smooth $(\kappa-1)$-dimensional surfaces, denoted here as $(\calZ_s)_{s=1}^{N_z}$, $1\leq N_z < \infty$, and a component with Hausdorff dimension at most $\kappa-2$, denoted here as $\calZ'$. That is, $\calZ = (\bigcup_{s=1}^{N_z} \calZ_s) \cup \calZ'$.
Each $\calZ_s$, $s=1,\ldots,N_z$ is contained in some indifference surface, which we denote here as $\calI_s$. Define the relative interior of $\calZ_s$ (with respect to $\calI_s$) as
$\ri \calZ_s := \{x\in Z_s:~ \exists \epsilon>0 \mbox{ s.t. } B(x,\epsilon)\cap\calI_s \subset \calZ_s\}$, and define the relative boundary of $\calZ_s$ as $\partial \calZ_s:=\cl \calZ_s \backslash \ri \calZ_s$.
We then define the relative boundary of $\calZ$ as
$$
\partial \calZ := \left(\bigcup_{s=1}^{N_z} \partial \calZ_s\right) \cup \calZ'.
$$
Note that $\partial \calZ$ is a set with Hausdorff dimension at most $\kappa-2$.
By Lemma \ref{lemma_calZ_normal} in the appendix, the BR-dynamics vector field is oriented tangentially along $\calZ$, in the sense that for any $x\in \calZ$ there holds $\nu\cdot y = 0$ for any vector $\nu$ normal to $\calZ$ at $x$, and any $y\in \FP(x)$.
This implies that BR paths can only enter or exit $\calZ$ through $\partial \calZ$.

Let
$$
X^* := X\backslash\left( Q\cup \calZ\right)
$$


\begin{example} \label{example3}
Consider the 3-player 2-action identical payoffs game with the (identical) utility function given in Figure \ref{example2_utility}
\renewcommand{\arraystretch}{1.2}
\begin{figure}[h]
    \centering
    \begin{subfigure}{.3\textwidth}
        \begin{tabular} {r|c|c|}
        \multicolumn{3}{c}{\hspace{1.5em} $A$\hspace{1.4em} $B$}\\
        \cline{2-3}
        $A$ & $1,~1$ & $0,~0$ \\
        \cline{2-3}
        $B$ & $0,~0$ & $2,~2$ \\
        \cline{2-3}
        \end{tabular}
        \caption{Player 3 plays $A$}
    \end{subfigure}
    \hspace{.5cm}
    \begin{subfigure}{.3\textwidth}
        \begin{tabular} {r|c|c|}
        \multicolumn{3}{c}{\hspace{1.5em} $A$\hspace{1.4em} $B$}\\
        \cline{2-3}
        $A$ & $3,~3$ & $0,~0$ \\
        \cline{2-3}
        $B$ & $0,~0$ & $6,~6$ \\
        \cline{2-3}
        \end{tabular}
        \caption{Player 3 plays $B$}
        \label{example2_utility}
    \end{subfigure}
    \caption{Utility structure for game in Example \ref{example3}}
    \vspace{-1.2em}
\end{figure}
\renewcommand{\arraystretch}{1.0}
We will refer to the column player as player 1, the row player as player 2, and the remaining player as player 3.
If player 3 plays action $A$ (respectively $B$), then the game is reduced to a $2\times 2$ game with BRD vector field shown in Figure \ref{fig:Ex3_subgame1} (Figure \ref{fig:Ex3_subgame2}), where the first subgame is familiar from Example \ref{example_FP1}.
The strategy space $X$ of the full game is a 3-dimensional cube---the BRD vector field for this game is visualized in Figures \ref{fig:Ex3_full}--\ref{fig:Ex3_full2}, where the blue arrows represent the vector field and the green surfaces represent indifference surfaces.
A plot of 70 BRD trajectories with random initializations is shown in Figure \ref{fig:Ex3_FP}.
Note that within the face $x_3 = 0$ (respectively, $x_3=1$) the vector field coincides with the $2\times 2$ vector field in Figure \ref{fig:Ex3_subgame1} (Figure \ref{fig:Ex3_subgame2}). Note also that the vector field jumps at the indifference surfaces.

The indifference surfaces are explicitly given by $\calI_i = \{x:~ x_i = \frac{1+5x_3}{3+6x_3} \}$, $i=1,2$ and $\calI_3 = \emptyset$ (since players have only two actions, we drop the additional sub-indices on $\calI$).
The game has one (incompletely) mixed equilibrium at $x^* = (\frac{2}{3},\frac{2}{3},1)$. The graph of the map $g$ associated with this equilibrium (see Section \ref{sec_diff_ineq_prelims}) coincides with the set $\calI_1\cap\calI_2$ (cf. Example \ref{example2}).
The set $Q$ is given by
$$
Q = \calI_1\cap\calI_2 = \bigg\{x:~ x_1 = x_2 = \frac{1+5x_3}{3+6x_3}\bigg\}.
$$
Note that this contains the set $\Lambda(x^*) = \Gr(g)$ and all points at which indifference surfaces intersect.

Figure \ref{fig:Ex3_side} shows a side view of the 3D BRD vector field. The surface $\calI_2$ is seen from this angle as the green curve. The BRD vector field is tangential to $\calI_2$ at any point $x$ with $(x_1,x_3) = (\frac{1}{2},\frac{1}{4})$, $x_2\geq \frac{1}{2}$. A similar situation holds for $\calI_1$.

From this we see that set $\calZ$ (the sub-manifold where trajectories may enter some indifference surface tangentially) is given by
$$
\calZ = \calZ_1\cup\calZ_2,
$$
where
$\calZ_1 = \{x:~ (x_2,x_3) := (\frac{1}{2},\frac{1}{4}),~x_1\geq \frac{1}{2}\}$ and $\calZ_2 := \{x:~ (x_1,x_3) = (\frac{1}{2},\frac{1}{4}),~x_2\geq \frac{1}{2}\}$.

{\tiny
\begin{figure}[h]
\centering
    \begin{subfigure}[b]{.38\textwidth}
        \centering
        \includegraphics[height=.85\textwidth]{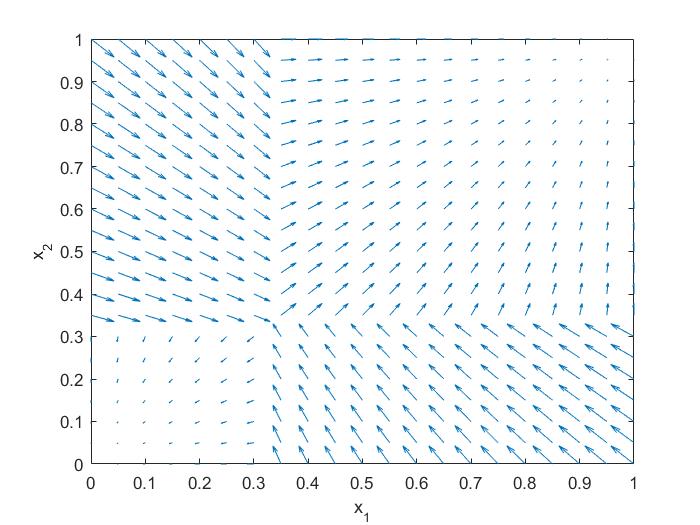}
        \caption{ \small Reduced $2\times 2$ BRD vector field along face $x_3=0$ in Example \ref{example3}.}
        \label{fig:Ex3_subgame1}
    \end{subfigure}
    \hspace{3em}
    \begin{subfigure}[b]{.38\textwidth}
        \centering
        \includegraphics[height=.85\textwidth]{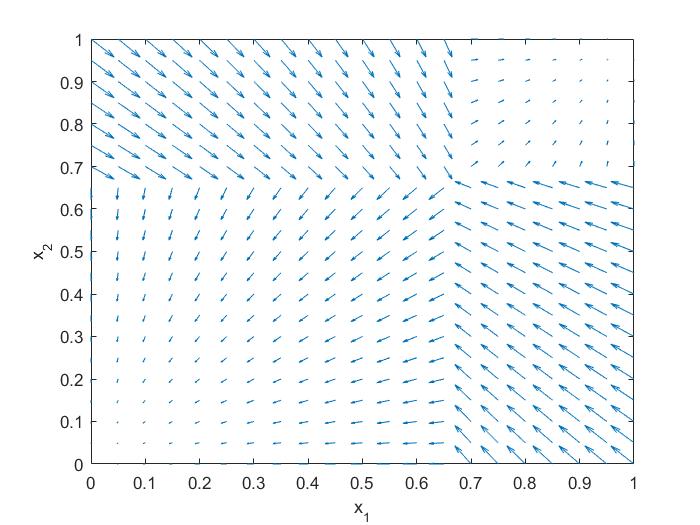}
        \caption{\small Reduced $2\times 2$ vector field along face $x_3=1$ in Example \ref{example3}.}
        \label{fig:Ex3_subgame2}
    \end{subfigure}
    \newline
    \begin{subfigure}[b]{.4\textwidth}
        \centering
        \includegraphics[height=.75\textwidth]{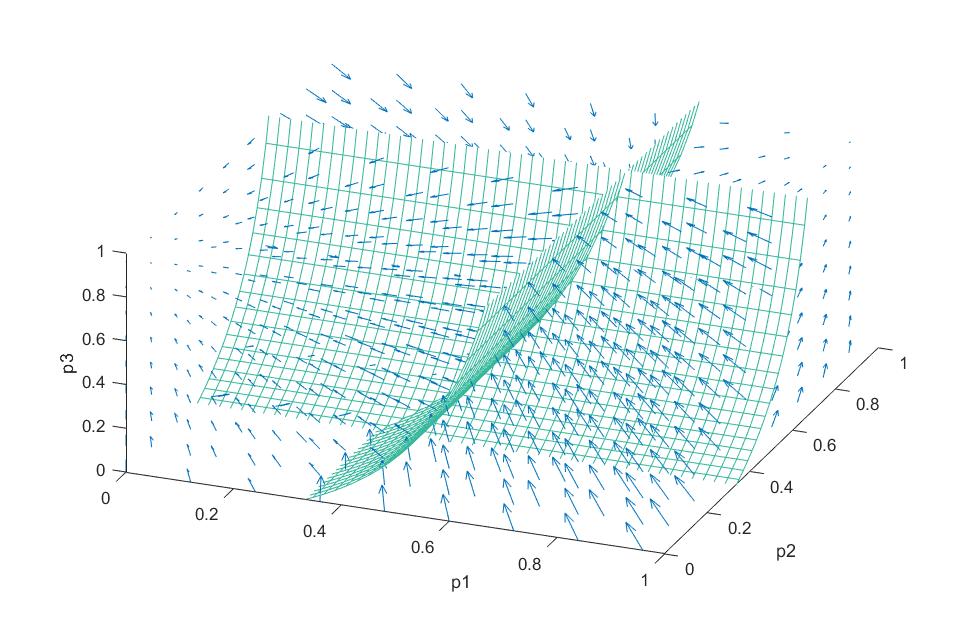}
        \caption{\tiny BRD vector field for game in Example \ref{example3}.}
        \label{fig:Ex3_full}
    \end{subfigure}
    \hspace{2em}
    \begin{subfigure}[b]{.4\textwidth}
        \centering
        \includegraphics[height=.75\textwidth]{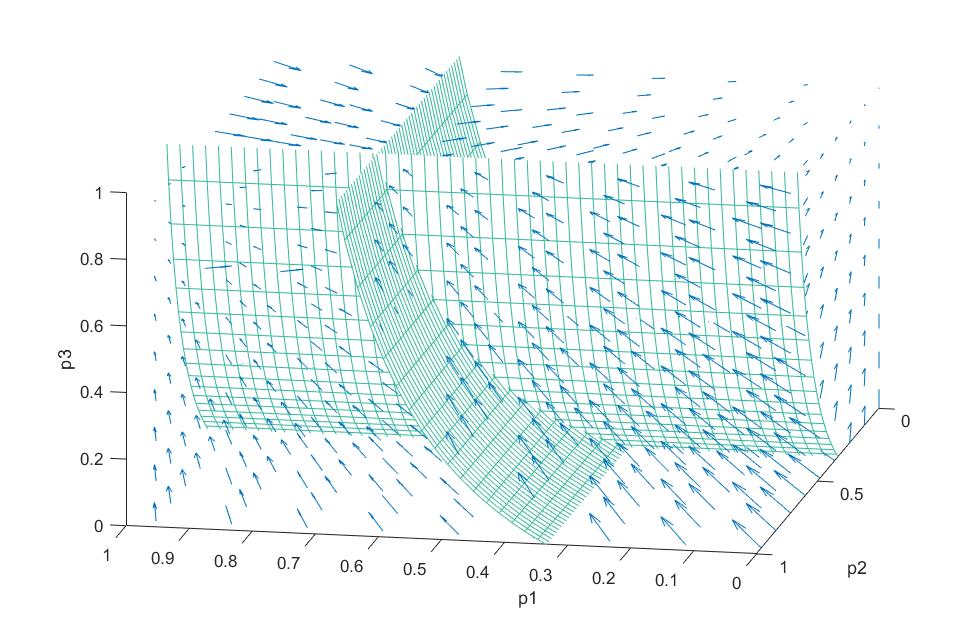}
        \caption{\small Alternate view of BRD vector field for game in Example \ref{example3}.}
        \label{fig:Ex3_full2}
    \end{subfigure}
    \newline
    \begin{subfigure}[b]{.4\textwidth}
        \centering
        \includegraphics[height=.75\textwidth]{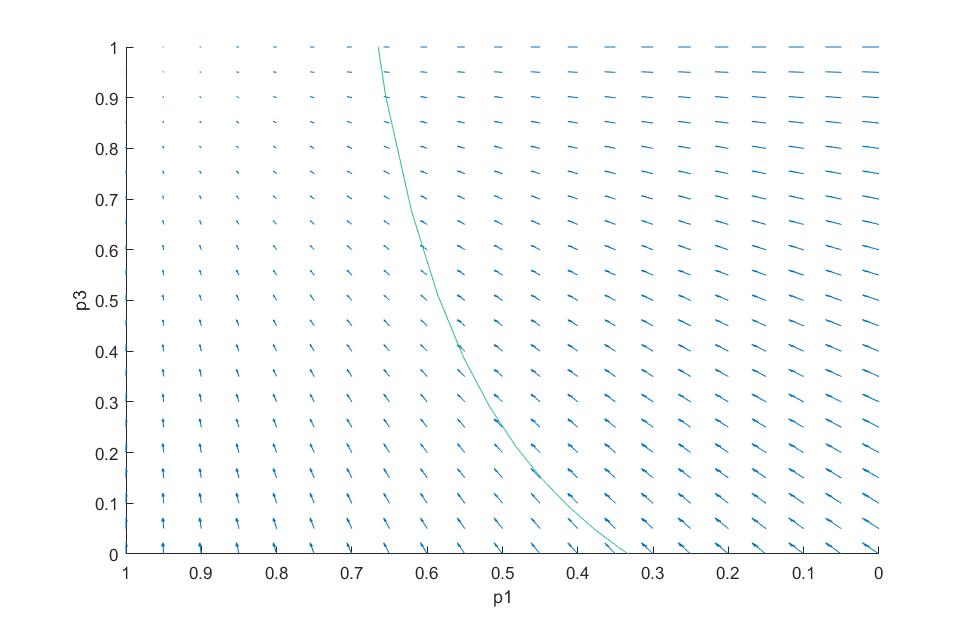}
        \caption{\small Side view of BRD vector field showing $\calI_2$ in Example \ref{example3}.}
        \label{fig:Ex3_side}
    \end{subfigure}
    \hspace{2em}
    \begin{subfigure}[b]{.4\textwidth}
        \centering
        \includegraphics[height=.75\textwidth]{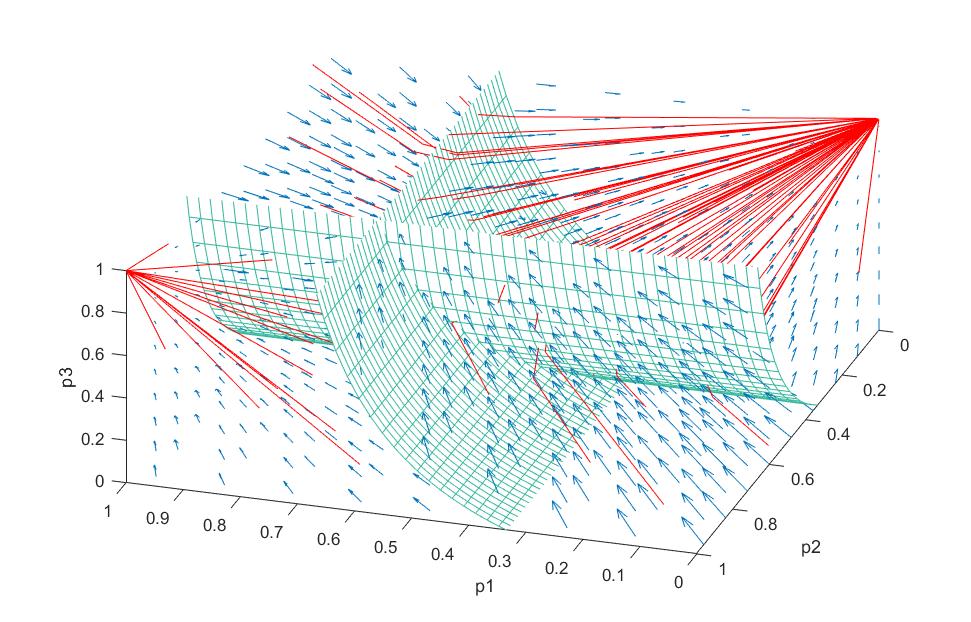}
        \caption{\small BRD trajectories in Example \ref{example3}. }
        \label{fig:Ex3_FP}
        \label{label4}
    \end{subfigure}
    \caption{}
\end{figure}

\vspace{-3em}

}

\end{example}

The following technical lemma will be used to show that the BR dynamics are well posed within $X^*$ (see Lemma \ref{lemma_unique_flow}). It is a consequence of the fact that the BR-dynamics vector field can only have jumps that are tangential to indifference surfaces.
\begin{lemma} \label{lemma_tangential_jumps}
Suppose $x\in X^*$ is in some indifference surface $\calI_{i,k,\ell}$. Then there exists a constant $c>0$ and a vector $\nu$ that is normal to $\calI_{i,k,\ell}$ at $x$, such that
$$
\nu\cdot z \geq c, \quad \forall ~ z\in \FP(\tilde x)
$$
for all $\tilde x \in X^*$ in a neighborhood of $x$.
\end{lemma}
\begin{proof}
By the definition of $\calI_{i,k,\ell}$, if $x\in \calI_{i,k,\ell}$ then for all $\hat x\in X$ such that $\hat x_{-i} = x_{-i}$ we have $\hat{x}\in \calI_{i,k,\ell}$. This implies that for any vector $\nu$ that is normal to $\calI_{i,k,\ell}$, the $(i,m)$-th component of $\nu$ must be zero for all $m=1,\ldots,K_i-1$.

Suppose that $x\in X^* \cap \calI_{i,k,\ell}$. Since $x\not\in Q$, there is a neighborhood of $x$ in which no indifference surface intersects with $\calI_{i,k,\ell}$.
This implies that for $\tilde x$ within a neighborhood of $x$, $\BR_{-i}(\tilde{x})=a_{-i}$ for some $a_{-i}$ that is a vertex of $X_{-i}$.

Together, these two facts imply that for all $\tilde x$ in a neighborhood of $x$, we have $\nu\cdot z' = \nu \cdot z''$ for all $z'\in \BR(x)$, $z''\in \BR(\tilde x)$, for any vector $\nu$ that is normal to $\calI_{i,k,\ell}$ at $x$.
Since $x\notin \calZ$, recalling the form of $\FP$ \eqref{def_FP_map}, this means we can choose a vector $\nu$ that is normal to $\calI_{i,k,\ell}$ at $x$ and a constant $c>0$ such that $\nu\cdot z > c$ for $z\in \FP(\tilde{x})$ for all $\tilde{x}$ in a neighborhood of $x$.
\end{proof}

The following lemma gives a well-posedness result for the BR dynamics inside $X^*$.
\begin{lemma} \label{lemma_unique_flow} For any $x_0\in X^*$, there exists a $T\in(0,\infty]$ and a unique absolutely-continuous function $\vx:[0,T]\to X^*$, with $\vx(0) = x_0$, solving the differential inclusion $\frac{d}{dt} \vx(t) \in \FP(\vx(t))$ for almost all $t\in[0,T]$.
\end{lemma}
\begin{proof}
If $x\in X^*$ is not on any indifference surface, then $\FP$ is single valued in a neighborhood of $x$, and \eqref{def_FP_autonomous} is (locally) a Lipschitz differential equation with unique local solution.

Suppose that $x_0\in X^*$ is on an indifference surface $\calI$.
By Lemma \ref{lemma_tangential_jumps}
there exists a constant $c>0$ such that for all $\tilde x$ in a neighborhood of $x$ we have $\FP(\tilde x)\cdot \nu>c$, where $\nu$ is a normal vector to $\calI$ at $x$.
This implies that
for $\delta>0$ sufficiently small we have $\{t\in[-\delta,\delta]:~\vx(t)\in \calI\} = \{0\}.$
Furthermore, since $x\not\in Q$, for $\delta>0$ sufficiently small we have
\begin{equation} \label{eq_zero_time_on_calI}
\{t\in[-\delta,\delta]:~\vx(t)\in \calI_{i,k,\ell}, \mbox{ for any } i,k,\ell\} = \{0\}.
\end{equation}

Now, let $x_0\in X^*$ and let $(\vx(t))_{t\geq 0}$ and $(\vz(t))_{t\geq 0}$ be two solutions to \eqref{def_FP_autonomous} with $\vx(0)=\vz(0)=x_0$. If $(\vx(t))_{t\geq 0}$ never crosses an indifference surface, then the flow is always classical and the two solutions always coincide; i.e., $\vx(t) = \vz(t),~t\geq 0$.
Suppose that $(\vx(t))_{t\geq 0}$ does cross an indifference surface and let $t^*\geq 0$ be first time when such a crossing occurs.
For $t<t^*$, the flow is classical and we have $\vx(t) = \vz(t)$ for $t\leq t^*$.

By \eqref{eq_zero_time_on_calI} we see that for $\delta >0$ sufficiently small, $\vx(t)$ is not in any indifference surface for $t \in [t^*-\delta,t^*+\delta]\backslash\{t^*\}$.
Suppose that at time $t=t^*+\delta$ we have $\vx(t) = \hat{x}\not= \hat{z}= \vz(t)$. Let $(\tilde{\vx}(\tau))_{\tau\geq 0}$ and $(\tilde{\vz}(\tau))_{\tau\geq 0}$ be solutions to the time-reversed BR-dynamics flow with $\tilde{\vx}(0) = \hat{x}$ and $\tilde{\vz}(0) = \hat{z}$.

Since $\hat{x}\not= \hat{z}$, and since the time-reversed flow is classical for $0\leq\tau<\delta$ (in particular, of the form $\dot{x} = a+x$ for some constant $a$), we get $\tilde{\vx}(\delta) \not= \tilde{\vz}(\delta)$. But this is impossible because
the paths $(\vx(t))_{t\geq 0}$ and $(\vz(t))_{t\geq 0}$ are absolutely continuous and
we already established that $\tilde{\vx}(\delta) = \vx(t^*) = \vz(t^*) = \tilde{\vz}(\delta)$.
\end{proof}

\begin{remark}
We emphasize that Lemma \ref{lemma_unique_flow} only shows uniqueness for a finite-time horizon. Uniqueness for an infinite-time horizon will be obtained at a later point (see Section \ref{sec_uniqueness}).
\end{remark}

Having established the well-posedness of BR dynamics in $X^*$ (and hence, almost everywhere in $X$) we will now proceed to show that, in some appropriate sense, the BR-dynamics vector field has bounded divergence (see Lemma \ref{lemma_FP_divergence}).
Of course, $\FP$ \eqref{def_FP_map} is a discontinuous set-valued function and the divergence of $\FP$ in the classical sense is not well defined. Instead, we will find it convenient to view $\FP$ as a function of bounded variation and consider an appropriate measure-theoretic notion of divergence for $\FP$.
With this in mind, we will now briefly introduce the notion of a function of bounded variation and an appropriate notion of divergence for such functions.

As a matter of notation, we say that $\lambda$ is a signed measure on $\R^\kappa$ if there exists a Radon measure $\mu$ on $\R^\kappa$ and a $\mu$-measurable function $\sigma:\R^\kappa \rightarrow \{-1, 1\}$ such that
\begin{equation} \label{def_signed_measure}
\lambda(K) = \int_K \sigma d\mu
\end{equation}
for all compact sets $K\subset \R^\kappa$. When convenient, we write $\sigma\mu$ to denote the signed measure $\lambda$ in \eqref{def_signed_measure}.

Letting elements $x\in X$ be written componentwise as $(x_s)_{s=1}^\kappa$, we recall \cite{EvansGariepy} that a function $u \in L^1(\Omega)$ (with $\Omega \subseteq \R^\kappa$, $\Omega$ open) is a function of bounded variation (i.e., a BV function) if there exist finite signed Radon measures $D_s u$ such that the integration by parts formula
\begin{equation}\label{eq_integration_by_parts}
\int_{\Omega} u \frac{\partial \phi}{\partial x_s} \dx = - \int_{\Omega} \phi \,d D_s u
\end{equation}
holds for all $\phi \in C^\infty_c(\Omega)$. The measure $D_s u$ is called the \emph{weak}, or  \emph{distributional}, \emph{partial derivative of $u$} with respect to $x_s$. We let $Du := (D_s u)_{s = 1,\dots, \kappa}$.

The measure $Du$ can be uniquely decomposed into three parts \cite{ambrosio2000functions} [cite other BV book]
\begin{equation}\label{eq_measure_decomposition}
Du = \nabla u \mathcal{L}^\kappa + Cu + Ju.
\end{equation}
Here $Ju$ is supported on a set $J_u$ with Hausdorff dimension $\kappa-1$, and $Cu$ is singular with respect to $\calL^\kappa$ and satisfies $Cu(E) = 0$ for all sets $E$ with finite $\calH^{\kappa-1}$ measure.

The $L^1$ function $\nabla u$ is analogous to a classical derivative, and in particular if $u$ is differentiable on an open set $V$ then $Du = \nabla u \mathcal{L}^\kappa$ on that set, with $\nabla u$ matching the classical derivative. Furthermore, if $u$ jumps across a smooth $(\kappa-1)$-dimensional hypersurface, then for $x$ on the hypersurface we have
\begin{equation} \label{eq_jump_decomposition}
Du = Ju = (u^+ - u^-) \nu d\calH^{\kappa-1},
\end{equation}
where $u^+$ is the value of $u$ on one side of the surface, $u^-$ is the value on the other, and $\nu$ is the normal vector pointing from $u^-$ to $u^+$ \cite{ambrosio2000functions}.

A vector-valued function $f \in L^1(\Omega : \R^\kappa)$ is a function of bounded variation if each of its components is also of bounded variation. Letting $f$ be written componentwise as $f=(f^s)_{s=1}^\kappa$, we write $Df := (D_j f^s)_{j,s = i, \dots, \kappa}$.

Next we define the divergence of a function $f \in L^1(\Omega : \R^\kappa)$, denoted by $D \cdot f$, as the measure
\[
D \cdot f := \sum_{s=1}^\kappa D_s f^s.
\]

Given a constant $c\in\R$, we say that $D \cdot f = c$ if $ D \cdot f = \frac{ d D \cdot f}{d \mathcal{L}^\kappa} \mathcal{L}^\kappa$, and $\frac{ d D \cdot f}{d \mathcal{L}^\kappa} = c$, where $\frac{ d D \cdot f}{d \mathcal{L}^\kappa}$ denotes the Radon-Nikodym derivative.
The following lemma characterizes the divergence of the BR-dynamics vector field. As a matter of notation, if a function $f:X \rightarrow X$ satisfies $f(x) \in \FP(x)$ for all $x\in X$ then we say $f$ is a \emph{selection} of $\FP$.
\begin{lemma} \label{lemma_FP_divergence}
For every selection $f$ of $\FP$, the vector field $f$ satisfies $D\cdot f = -1$.
\end{lemma}
The proof of this lemma follows from the fact that $\FP$ is piecewise linear, and any jumps in $\FP$ are tangential to indifference surfaces.
\begin{proof}
Suppose $f$ is a selection of $\FP$, and let $f$ be written componentwise as \newline $f = (f_i^k)_{\substack{i=1,\ldots,N,\\ k=1,\ldots,K_i-1}}$.
Let $i\in\{1,\ldots,N\}$ and $k\in\{1,\ldots,K_i-1\}$.
Let $D_{j,\ell} f_i^k$ denote the weak partial derivative of $f_i^k$ with respect to $x_j^\ell$, $j=1,\ldots,N$, $\ell=1,\ldots,K_j-1$, and let $D f_i^k = \left( D_{j,\ell} f_i^k \right)_{j=1,\ldots,N,~\ell=1,\ldots,K_j-1}$.  Let $J f_i^k = \left(J_{j,\ell} f_i^k\right)_{j=1,\ldots,N,~\ell=1,\ldots,K_j-1}$ denote the jump component associated with $D f_i^k$ (see \eqref{eq_measure_decomposition}).

The vector field $f$ is piecewise linear. Breaking up $f$ over regions in which it is linear we see that $\frac{d D\cdot f}{d\calL^{\kappa}} = -1$. It remains to show that $D\cdot f$ has no singular component; i.e., under the decomposition \eqref{eq_measure_decomposition}, the measure $D\cdot f$ has zero Cantor component and zero jump component.

Since $f_i^k$ is piecewise linear and only jumps on the set $\bigcup_{\ell=1,~\ell\not=k}^{K_i} \calI_{i,k,\ell}$ which has finite $\kappa-1$ measure, $f_i^k$ has no Cantor part; that is, $C f_i^k = (C_{j,\ell} f_i^k)_{\substack{j=1,\ldots,N,~ \ell=1,\ldots,\gamma_j}} = 0$ (see \eqref{eq_measure_decomposition}).
Hence, the singular component of $D\cdot f$, which we denote here as $S$, has no Cantor part and is given by $S:=\sum_{i=1}^N \sum_{k=1}^{K_i-1} J_{i,k}f_i^k$.

Suppose that $x\in \calI_{i,k,\ell}$ for some $\ell$ (recall $\ell\not= k$). Suppose $\nu$ is a vector that is normal to $\calI_{i,k,\ell}$ at $x$.
By the definition of $\calI_{i,k,\ell}$, if $x\in \calI_{i,k,\ell}$ then for all $\hat x\in X$ such that $\hat x_{-i} = x_{-i}$ we have $\hat{x}\in \calI_{i,k,\ell}$. This implies that the $(i,k)$-th component of $\nu$ must be zero.
Since $J f_i^k=((f_i^k)^+ - (f_i^k)^-)\nu\calH^{\kappa-1}$ for $x$ on $\bigcup_{\ell=1,\ell\not=k}^{K_i} \calI_{i,k,\ell}$ (see \eqref{eq_jump_decomposition}), taking the $(i,k)$-th component we get $J_{i,k}f_i^k\calH^{\kappa-1} = 0$.

Since this is true for every pair $(i,k)$ we see that $S=0$, and hence $D\cdot f = -1$ in the interior of $X$.
An identical argument holds on the boundary of $X$, and hence, $S=0$ and $D\cdot f = -1$.
\end{proof}

The following lemma shows that for sets $E\subseteq X^*$ with relatively smooth boundary, the surface integral of $\FP$ over the boundary of $E$ is well defined.
\begin{lemma} \label{lemma_div_thrm}
Let $E$ be a subset of $X^*$ with piecewise smooth boundary.
For any functions $f,g$ that are selections of $\FP$ we have
$$
\int_{\partial E} f\cdot \nu_E  d\calH^{\kappa-1} = \int_{\partial E} g\cdot \nu_E  d\calH^{\kappa-1} =: \int_{\partial E} \FP\cdot \nu_E  d\calH^{\kappa-1},
$$
where $\nu_E$ denotes the outer normal vector of $E$.
\end{lemma}
\begin{proof}

Suppose $x\in X^*$ is not on any indifference surface $\calI_{i,k,\ell}$. Then $\FP(x)$ maps to a singleton and $f(x) = g(x)$.

Suppose $x\in X^*$ is on an indifference surface $\calI_{i,k,\ell}$. Let $\nu_{\calI}$ denote a normal vector to $\calI_{i,k,\ell}$. Since $x\in X^*$, the vector field $\FP$ can only jump tangentially to $\nu_{\calI}$. Using similar reasoning to the proof of Lemma \ref{lemma_tangential_jumps}, this implies that for any $a,b\in \FP(x)$ we have $a\cdot \nu_{\calI}(x) = b\cdot \nu_{\calI}(x)$. Hence $\FP(x)\cdot \nu:= a\cdot \nu,~a\in \FP(x)$ is well defined for such $x$.

In particular, note that if $x\in X^*$ is on some indifference surface $\calI$ and $\nu_{\calI} = \nu_{E}$ at $x$, then $f(x)\cdot \nu_{E}=\FP(x)\cdot \nu_{\calI}$ for any function $f$ that is a selection of $\FP$.

Let $\widehat{\calI}$ be the union of all indifference surfaces.
Since $\partial E$ is piecewise continuous and the indifference surfaces are smooth, the set
$S:=\{x\in X^*:~ x\in \widehat{\calI}\cap \partial E,~\nu_{\widehat{\calI}}(x) \not= \nu_{\partial E}(x)\}$
has $\calH^{\kappa-1}$-measure zero, where $\nu_{\widehat{\calI}}(x)$ and $\nu_{\partial E}(x)$ denote the normal vectors to $\widehat{\calI}$ and $\partial E$ at $x$.

We have shown that $f\big\vert_{(\partial E)\backslash S} = g\big\vert_{(\partial E)\backslash S}$ for any selections $f,g$ of $\FP$, and $\calH^{\kappa-1}(S)=0$, and hence,
\begin{align*}
\int_{\partial E} f\cdot \nu_E d\calH^{\kappa-1}
= \int_{\partial E} g\cdot \nu_E d\calH^{\kappa-1}
\end{align*}
for any selections $f,g$ of $\FP$.
\end{proof}

The following lemma shows that, within $X^*$, the BR-dynamics vector field compresses mass at a rate of $-1$.
In particular, this implies that, within $X^*$, BR dynamics cannot map a set of positive measure to a set of zero measure in finite time.\footnote{We note that this result can also be derived as a consequence of Lemma 3.1 in \cite{chen-ziemer2009gauss}. For the sake of completeness and to simplify the presentation, we give a proof of the result here using the notation and tools introduced in the paper.}
\begin{lemma} \label{eq_mass_concentration}
Let $E$ be a compact subset of $X^*$ with piecewise smooth boundary and finite perimeter. Then
\begin{equation} 
\int_{\partial E} \FP\cdot\nu_{E}\,d\calH^{\kappa-1} = -\calL^\kappa(E),
\end{equation}
where $\nu_E$ denotes the outer normal vector of $E$.
\end{lemma}
\begin{proof}
We first note that by Lemma \ref{lemma_FP_divergence} for every selection $f$ of $\FP$ we have $\int_E dD\cdot f = -\calL^{\kappa}(E)$.

Let $(f_n)_{n\geq 1}$, $f_n:X^*\rightarrow X^*$ be a sequence of uniformly bounded $C^1$ functions such that $f_n\rightarrow f$ a.e. for some function $f:X^*\rightarrow X^*$ satisfying $f(x)\in \FP(x)$ for all $x\in X^*$. (Such a sequence can be explicitly constructed by smoothing the BR-dynamics vector field, e.g., \cite{Fud92}.)

Let $f$ and each $f_n$ be written componentwise as $f = (f^s)_{s=1}^\kappa$ and $f_n = (f_n^s)_{s=1}^\kappa$. Let $D\cdot f_n = \sum_{s=1}^\kappa D_{s} f_n^{s}$ be the divergence measure associated with $f_n$ and $D\cdot f$ $= \sum_{s=1}^\kappa D_{s} f^{s}$ the divergence measure associated with $f$.
Since $f$ and $f_n$ are BV functions, by \eqref{eq_integration_by_parts} we have
$$
-\int_{X^*} f_n^s \frac{\partial \phi}{\partial x_s}\,dx = \int_{X^*} \phi \,D_s f_n^s, ~~\quad \mbox{and} ~~\quad
-\int_{X^*} f^s \frac{\partial \phi}{\partial x_s}\,dx = \int_{X^*} \phi \,D_s f^s
$$
for $n\in \N$, $s=1,\ldots,\kappa$, for any $\phi\in C^1_c(X^*)$.

For a function $\phi\in C^1_c(X^*)$, there exists a constant $c>0$ such that $|\frac{\partial \phi(x)}{\partial x_s}| < c$ for all $x\in X^*$. Since $(f_n)_{n\geq 1}$ is uniformly bounded, $|f_n(x) \frac{\partial \phi(x)}{\partial x_s}|$ is bounded by some constant $c>0$ for all $x\in X^*$,
and since $X^*$ is a bounded set, the constant function $c\rchi_{X^*}$ (which dominates $|f_n \frac{\partial \phi}{\partial x_s}|$ on $X^*$) is integrable.
Noting that $f_n \frac{\partial \phi}{\partial x_s} \rightarrow f\frac{\partial \phi}{\partial x_s}$ pointwise, the dominated convergence theorem gives
\begin{align}
\lim_{n\rightarrow \infty} \int_{X^*} \phi \,D_s f_n^s =  -\lim_{n\rightarrow \infty} \int_{X^*} f_n^s \frac{\partial \phi}{\partial x_s}\dx = -\int_{X^*} f^s \frac{\partial \phi}{\partial x_s}\,dx = \int_{X^*} \phi \,D_s f^s.
\end{align}
for $n\in \N$, $s=1,\ldots,\kappa$.
This implies that the sequence of measures $(D\cdot f_n)_{n\geq 1}$ converges weakly to $D\cdot f$ in the sense that for any $\phi\in C^1_c(X^*)$ there holds $\lim\limits_{n\rightarrow\infty} \int_{X^*} \phi \, dD\cdot f_n = \int_{X^*} \phi \, dD\cdot f$.
Letting $\phi$ approximate the characteristic function $\rchi_{E}$, and noting that by Lemma \ref{lemma_FP_divergence} we have $(D\cdot f)(\partial E) = 0$, we see that
$\lim_{n\rightarrow \infty} \int_{E} dD\cdot f_n = \int_{E} dD\cdot f$.
Hence,
\begin{align*}
-\calL^{\kappa}(E) & = \int_{E} dD\cdot f\\
& = \lim_{n\rightarrow \infty} \int_{E} dD\cdot f_n\\
& = \lim_{n\rightarrow \infty} \int_{\partial E} f_n\cdot \nu_E d\calH^{\kappa-1}\\
& = \int_{\partial E} f\cdot \nu_E d\calH^{\kappa-1}\\
& = \int_{\partial E} \FP\cdot \nu_E d\calH^{\kappa-1},
\end{align*}
where the third line follows from the Gauss-Green theorem \cite{EvansGariepy}, the fourth line follows from the dominated convergence theorem (by assumption, $E$ has finite perimeter and a piecewise smooth boundary, and $f$ is bounded), and the fifth line follows from Lemma \ref{lemma_div_thrm}.
\end{proof}
We now prove Proposition \ref{prop_finite_time_convergence}.
\begin{proof}
We begin by noting that, by Lemma \ref{lemma_clQ_dim} in the appendix, $\cl Q$, has Hausdorff dimension at most $\kappa-2$.

Let $\epsilon>0$. By the definition of the Hausdorff measure (\hspace{-.004cm}\cite{EvansGariepy}, Chapter 2), there exists a countable collection of balls $(B_\epsilon^j)_{j\geq 1}$, each with diameter less than $\epsilon$, such that $ \cl Q\cup\partial \calZ  \subset \bigcup_{j\geq 1} B_\epsilon^j$ and $\sum_{j=1}^\infty c\left(\frac{\diam B_\epsilon^j}{2}\right)^{\kappa-2} < 2\calH^{\kappa-2}(\cl Q\cup\partial\calZ)$, where $c:=\frac{\pi^{\kappa-2}}{\Gamma(\frac{\kappa-2}{2})+1}$, and where $\Gamma$ in this context denotes the standard $\Gamma$ function.

Since $\partial \calZ$ is closed, $\cl Q\cup \partial Z$ is closed, and hence there exists a finite subcover $(B_\epsilon^j)_{j=1}^{N_\epsilon}$ such that $ \cl Q\cup\partial \calZ \subset \bigcup_{j=1}^{N_\epsilon} B_\epsilon^j $. Let $B_\epsilon := \bigcup_{j=1}^{N_\epsilon} B_\epsilon^j$, and let
$$
X^*_\epsilon := X\backslash \left(B_\epsilon \cup \calZ\right).
$$
Note that, by Lemma \ref{lemma_ball_boundary} in the appendix we have
\begin{equation} \label{eq_boundary_shrinks}
\lim_{\epsilon \to 0} \calH^{\kappa-1}(\partial B_\epsilon) = 0.
\end{equation}

Fix some time $T> 0$, and for $0< t\leq T$, let
\begin{align*}
E_\epsilon(T-t) := \{x_0\in X^*_\epsilon:~\vx(0) = x_0,~ & \vx(t) \mbox{ is a BR process},\\
&\vx(s)\in B_\epsilon, \mbox{ for some }0<s\leq t\}
\end{align*}
and note that the boundary $\partial E_\epsilon(T-t)$ is piecewise smooth. The set $E_\epsilon(T-t)$ may be thought of as the set obtained by tracing paths backwards out of $B_\epsilon$ from time $T$ back to time $T-t$. Let
$$
V_\epsilon(t):= \calL^\kappa(E_\epsilon(T-t)).
$$

Letting $R_\epsilon$ denote the flux through $\partial B_\epsilon$ into $E_\epsilon(T-t)$ and again letting $\nu$ denote the outer normal to $\partial E_\epsilon(T-t)$, for $t>0$ we have
\begin{align} \label{prop_finite_time_eq1}
\frac{d}{dt} V_\epsilon(t) & = \int_{\partial E_\epsilon(T-t)\backslash \partial B_\epsilon } -\FP\cdot\nu \dx \\
\nonumber & \leq R_\epsilon + \int_{\partial E_\epsilon(T-t)}-\FP\cdot\nu \dx \\
\nonumber & \leq R_\epsilon + \calL^\kappa(E_\epsilon(T-t))\\
\nonumber & = R_\epsilon + V_\epsilon(t),
\end{align}
where the third line follows by Lemma \ref{eq_mass_concentration}.

Noting that $\|\FP\|_{\infty} < \infty$, the flux through $\partial B_\epsilon$ is bounded by
\begin{equation}\label{prop_finite_time_eq2}
R_\epsilon \leq \calH^{\kappa-1}(\partial B_\epsilon)\|\FP\|_{\infty} =: \bar{R}_\epsilon.
\end{equation}
By \eqref{eq_boundary_shrinks} we have $\calH^{\kappa-1}(\partial B_\epsilon)\rightarrow 0$ as $\epsilon \rightarrow 0$, and hence
$\bar{R}_\epsilon \rightarrow 0 \mbox{ as } \epsilon \rightarrow 0.$

Using the integral form of Gronwall's inequality, \eqref{prop_finite_time_eq1} and \eqref{prop_finite_time_eq2} give $V_\epsilon(t) \leq t\bar{R}_\epsilon e^{t}$, $0<t\leq T$.
In particular, this means that
\begin{equation} \label{eq_gronwall_bound}
\calL^{\kappa}(E_\epsilon(0)) \leq \bar{R}_{\epsilon}e^T,
\end{equation}
where the right hand side goes to zero as $\epsilon \rightarrow 0$.
Sending $\epsilon\rightarrow 0$, we see that the set $W(T) :=\{x_0\in X^*:~\vx(0) = x_0,~ \vx(t) \mbox{ is a BR process},~\vx(s) \in Q\cup\partial\calZ \mbox{ for some } 0< s\leq T \}$ has  $\calL^\kappa$-measure zero.

Since paths may only enter $\calZ$ through the boundary $\partial \calZ$, this means that the set of points in $X$ from which $\calZ$ can be reached within time $T$ is contained in $W(T)\cup\calZ$. Furthermore, the set of points from which $Q\cup\calZ$ can be reached within time $T$ is contained in $W(T)\cup\calZ \cup Q$, which is a $\calL^{\kappa}$-measure zero set. Since this is true for every $T> 0$, we get the desired result.
\end{proof}

\subsection{Uniqueness of Solutions in Potential Games} \label{sec_uniqueness}
Solutions of \eqref{def_FP_autonomous} are known to always exist (see Section \ref{sec_BR_dynamics_def}). However, being a differential inclusion, solutions of \eqref{def_FP_autonomous} may not always be unique (see Section \ref{sec_example}).
In this Section we show that, although not always unique, solutions of \eqref{def_FP_autonomous} are \emph{almost always} unique in potential games (i.e., we prove Proposition \ref{prop_uniqueness}).


This issue can be readily addressed using the arguments above.
Note the following:
\begin{itemize}
\item The proof of Lemma \ref{lemma_unique_flow} shows that solution curves with initial conditions in $X^*$ are  unique so long as they remain in $X^*$.
\item The proof of Proposition \ref{prop_finite_time_convergence} shows that the set $Q\cup\calZ$ (or equivalently, the set $X\backslash X^*$) can only be reached in finite time from a $\calL^\kappa$-measure zero subset of initial conditions in $X^*$.
\end{itemize}
Since $X^* := X\backslash (Q\cup\calZ)$, this implies that for almost every initial condition in $X^*$, solutions remain in $X^*$ for all $t\geq 0$ and such solutions are unique for all $t\geq 0$. Since $\calL^\kappa(X\backslash X^*)=0$, we see that for almost every initial condition in $X$ there exists a unique solution of \eqref{def_FP_autonomous} and the solution is defined for all $t\geq 0$.
Recalling that we have assumed throughout the section that the game $\Gamma$ is regular, this proves Proposition \ref{prop_uniqueness}.

\begin{remark} [Viewing \eqref{def_FP_autonomous} as a differential equation] \label{remark_BR_diff_equation}
As usual, suppose a potential game is regular. The proof of Lemma \ref{lemma_unique_flow} shows that if a solution curve $\vx$ resides in $X^*$ over some time interval $[0,T]$ then $\BR(\vx(t)) - \vx(t)$ is single-valued for a.e. $t\in [0,T]$. (The map $\BR(x)$ is single valued except for $x$ on indifference surfaces. But the proof of Lemma \ref{lemma_unique_flow} shows that, while in $X^*$, any solution $\vx$ crosses all indifference surfaces instantly.) Furthermore, as discussed above, the proof of Proposition \ref{prop_finite_time_convergence} implies that for a.e. initial condition in $X^*$ (and hence, a.e. initial condition in $X$) solutions remain in $X^*$ for all $t\geq 0$. Thus, for a.e. initial condition in $X$, the vector field $\BR(\vx(t)) - \vx(t)$ is single valued along the solution curve $\vx(t)$ for a.e. $t\geq 0$. This justifies the remark in the introduction that, in potential games, it is relatively safe to think of \eqref{def_FP_autonomous} as a differential equation (with discontinuous right-hand side).
\end{remark}

\subsection{Infinite-Time Convergence} \label{sec_infinite_time_conv}
%

The following proposition shows that it is not possible to converge to $\Lambda(x^*)$ in infinite time.
\begin{proposition}\label{prop_infinite_time_convergence}
Let $\Gamma$ be a regular potential game and let $x^*$ be a mixed-strategy equilibrium. Suppose $(\vx(t))_{t\geq 0}$ is a BR process and $\vx(t)\rightarrow x^*$. Then $\vx(t)$ converges to $\Lambda(x^*)$ in finite time.
\end{proposition}
\begin{proof}
Without loss of generality, assume that for all $t\geq 0$, $\vx(t)$ is sufficiently close to $x^*$ so that \eqref{equation_pot_inequality1_1} and \eqref{equation_pot_inequality2_1} hold.
From the definitions of $\Lambda(x^*)$ and $\calP$ we see that
\begin{equation}
\vx(t) \rightarrow \Lambda(x^*) \quad \quad \iff \quad \quad \calP(\vx(t)) \rightarrow \calP(x^*).
\end{equation}
If we integrate \eqref{equation_pot_inequality2_1}, use the fact $\calP(\vx(t))$ \hspace{-4pt} $\to$ \hspace{-4pt}  $\calP(x^*)$, and set
$e(t) := d(\calP(\vx(t)),\calP(x^*))$, then we find that
\begin{equation} \label{eq_inf_time_pf_eq1}
\tilde U(\calP(x^*)) - \tilde U(\calP(\vx(t))) \geq c_2\int_{t}^{\infty} e(s)ds.
\end{equation}
Using \eqref{equation_pot_inequality1_1} above we get
\begin{equation} \label{eq_inf_time_pf_eq2}
c e^2(t) \geq  \int_t^\infty e(s) \,ds,
\end{equation}
with $c = c_1/c_2$.
Let $\eta>0$ and suppose that for some time $t$ we have $e(t) \leq \eta$.
Using Markov's inequality and applying \eqref{eq_inf_time_pf_eq2} we can bound the time spent in a ``shell'' near $\Lambda(x^*)$ as
\begin{align*}
\calL^1\left(\{s: \eta \geq e(s) > \eta/2 \} \right) & \leq \frac{2}{\eta} \int_t^\infty e(s)\ds\\
& \leq \frac{2}{\eta} c e^2(t)\\
& \leq 2c\eta.
\end{align*}
Without loss of generality, assume that $e(0) \geq e(t)$ for $t\geq 0$. Repeatedly applying the above inequality we get
\begin{align*}
\calL^1\left(\{s: e(s) > 0, ~s\geq 0\} \right) & = \sum_{k\geq 0} \calL^1\left(\big\{s: \frac{e(0)}{2^k} \geq e(s) > \frac{e(0)}{2^{k+1}}  \big\} \right)\\
& \leq \sum_{k\geq 0} 2c \frac{e(0)}{2^k}\\
& \leq 4ce(0).
\end{align*}
%
%
Thus if $\calP(\vx(t))$ converges to $\calP(x^*)$, it must reach it for the first time in finite time.

By construction $\calP(x) = \calP(x^*)$ if and only if $x\in \Lambda(x^*)$. Hence, if $\vx(t)$ converges to $\Lambda(x^*)$ it must reach it for the first time in finite time.

By \eqref{equation_pot_inequality2_1} we have $\frac{d}{dt}\tilde U(\calP(\vx(t))) \geq 0$ in a neighborhood of $\calP(x^*)$. Since $\Gamma$ is non-degenerate, the Hessian of $\tilde U$ is invertible at $\calP(x^*)$, and for all $\tilde x\in \tilde X$ in a punctured ball around $\calP(x^*)$ we have $\tilde U(\tilde x) \not= \tilde U(\calP(x^*))$.
Thus, if $\vx(t) \rightarrow x^*$ and $\calP(\vx(T)) = \calP(x^*)$ (i.e., $\vx(T) \in \Lambda(x^*)$) for some $T\geq 0$, then we must have $\calP(\vx(t)) = \calP(x^*)$ (i.e., $x(t) \in \Lambda(x^*)$) for all $t\geq T$.
Contrariwise, we would have $\tilde U(\calP(x^*)) =  \tilde U(\calP(\vx(T))) < \lim_{s\rightarrow \infty } \tilde U(\calP(\vx(s))) = \tilde U(\calP(x^*))$, which is a contradiction.
\end{proof}

\section{Convergence Rate Bound}\label{sec_conv_rate}
In this section we will prove Theorem \ref{thrm_conv_rate} as a simple consequence of Theorem \ref{thrm_main_result}.  More precisely, we will prove the following proposition which implies Theorem \ref{thrm_conv_rate}.
\begin{proposition} \label{prop_FP_conv_rate1}
Let $\Gamma$ be a regular potential game. Then:\\
\noindent (i) For \emph{almost every} initial condition $x_0\in X$, there exists a constant $c=c(\Gamma,x_0)$ such that if $\vx$ is a BR process associated with $\Gamma$ and $\vx(0) = x_0$, then
\begin{equation} \label{eq_FP_conv_rate1}
d(\vx(t),NE) \leq ce^{- t}.
\end{equation}
\noindent (ii) For every BR process $\vx$, there exists a constant $c=c(\Gamma,\vx)$ such that \eqref{eq_FP_conv_rate1} holds.
\end{proposition}
Part (i) of the proposition states that for almost every initial condition $x_0$, the constant $c$ in \eqref{eq_FP_conv_rate1} is uniquely determined by the game $\Gamma$ and the initial condition $x_0$. Part (ii) of the proposition allows one to handle BR processes starting from initial conditions where uniqueness of solutions may fail. In particular, part (ii) shows that if you allow the constant to depend on the solution $\vx$ rather than the initial condition
then the rate of convergence is always (asymptotically) exponential. However, we emphasize that part (ii) makes a somewhat weaker statement than part (i) since the constant $c$ in part (ii) can be made arbitrarily large in any potential game by allowing a solution $\vx$ to rest at a mixed equilibrium for an arbitrary length of time before moving elsewhere.\footnote{Harris (\hspace{-.001em}\cite{harris1998rate}, Conjecture 25) conjectured part (ii) of Proposition \ref{prop_FP_conv_rate1}. Using Theorem \ref{thrm_main_result} and Proposition \ref{prop_uniqueness} we are able to resolve Harris's conjecture and prove the slightly stronger result of part (i) for almost every initial condition.}
\begin{remark}
In the above proposition, it is possible to make the constant $c$ arbitrarily large by bringing the game $\Gamma$ arbitrarily close to the set of irregular potential games. For example, this was done in \cite{brandt2010rate} in order to achieve arbitrarily slow convergence in fictitious play in potential games. In future work we intend to address this issue by studying uniform bounds on the constant $c$ in \eqref{eq_FP_conv_rate1} for all potential games $\Gamma$ with distance at least $\delta>0$ from the set of irregular games.
\end{remark}

In order to prove Proposition \ref{prop_FP_conv_rate1}, we will require the following auxiliary lemma.
\begin{lemma}\label{lemma_strictness}
Let $x^*\in X$ be a pure-strategy equilibrium of a regular potential game. Then for all $x\in X$ in a neighborhood of $x^*$ there holds
$
\BR(x) = \{x^*\};
$
that is, the pure-strategy equilibrium $x^*$ is the unique best response to every $x$ in a neighborhood of $x^*$.
\end{lemma}
This lemma follows readily from the observation that in regular potential games, all pure NE are strict.\footnote{Every regular equilibrium is quasi-strict \cite{van1991stability}, and a pure-strategy equilibrium is quasi-strict if and only if it is strict. Hence, in regular potential games, all pure NE are strict.}
We will now prove Proposition \ref{prop_FP_conv_rate1}.
\begin{proof}
Theorem \ref{thrm_main_result} and Proposition \ref{prop_uniqueness} imply that
there exists a set $\Omega \subset X$ satisfying the following properties: (a) $\calL^\kappa(X\backslash \Omega)=0$, (b) for every BR process $\vx$ with initial condition $x_0\in\Omega$, $\vx$ is the unique BR process satisfying $\vx(0) = x_0$, and $\vx$ converges to a pure-strategy NE.

Let $x_0\in \Omega$, let $\vx$ be a BR process with $\vx(0) = x_0$, and let $x^*$ be the pure-strategy NE to which $\vx$ converges.
Without loss of generality, assume that the pure-strategy set $Y$ is reordered so that
\begin{equation}\label{eq_x_is_zero}
x^*=0;
\end{equation}
(i.e., $T_i^1(x^*_i) = 1$ for all $i=1,\ldots,N$, where $T_i^k$ is defined as in Section \ref{sec_prelims}).

By Lemma \ref{lemma_strictness}, for all $x$ in a neighborhood of $x^*$ we have $\BR(x) = x^*$. Since $\vx(t) \rightarrow x^*$, this, along with \eqref{def_FP_autonomous} and \eqref{eq_x_is_zero}, implies that
there exists a time $\tau=\tau(\Gamma,x_0)>0$ such that for all $t\geq \tau$,  we have
$\dot{\vx}(t) = - \vx(t)$.
Hence, for $t\geq \tau$ we have $\|\vx(t)\| = \|\vx(\tau)\|e^{\tau-t}$.
Letting $c := \sup_{t\in [0,\tau]} \|\vx(t)\|e^{\tau}$ 
we get $\|\vx(t)\| \leq ce^{-t}$ for all $t\geq 0$. This proves part (i) of the proposition.

To prove part (ii) of the proposition, we only need consider initial conditions $x_0\in X\backslash \Omega$. Suppose $\vx(0) = x_0$ and $\vx$ converges to a pure NE. Then using the same reasoning as above, there exists a time $\tau = \tau(\Gamma,\vx)>0$ such that $\dot{\vx}(t) = - \vx(t)$ for all $t\geq \tau$. As before, letting $c := \sup_{t\in [0,\tau]} \|\vx(t)\|e^{\tau}$ we get the desired result. On the other hand, if $\vx$ converges to a mixed equilibrium, then by Proposition \ref{prop_infinite_time_convergence} it does so in finite time. This proves part (ii) of the proposition.
\end{proof}
\begin{remark}
In Examples \ref{example_FP1}, \ref{example2}, and \ref{example3} one observes that from almost every initial condition, solution curves $\vx$ of \eqref{def_FP_autonomous} eventually enter a region where the best response settles on some pure strategy $x^*$; i.e., $\BR(\vx(t)) = x^*$, for all $t\geq T$ for some $T\geq 0$. From here BR dynamics assume the form $\dot \vx(t) = x^* - \vx(t)$, for all $t\geq T$, which is linear, and hence converges at an exponential rate.
\end{remark} 

\section*{Appendix}
\begin{lemma}\label{lemma_two_players_mixing}
Suppose $\Gamma$ is a regular game. At any mixed equilibrium there are at least two players using mixed strategies.
\end{lemma}
\begin{proof}
Suppose that $x^*$ is an equilibrium in which only one player uses a mixed strategy---say, player 1. Let $C_i = \carr_i(x^*)$ and $\gamma_i = |C_i|$. Then the mixed strategy Hessian is given by
$\tilde \vH(x^*) =(\frac{\partial^2 U(x^*)}{\partial x_1^k \partial x_1^\ell})_{k,\ell=1,\ldots,\gamma_i} = 0$, (note the subscripts of 1) where the equality to zero follows since $U$ is linear in $x_1$. But this implies that $x^*$ is a second-order degenerate equilibrium, which contradicts the regularity of $\Gamma$.
\end{proof}

\begin{lemma} \label{lemma_apx_BR_vs_gradient}
Let $x\in X$ and $i=1,\ldots,N$. Assume $Y_i$ is ordered so that $\actionione\in BR_i(x_{-i})$. Then:\\
(i) For $k=1,\ldots,K_i-1$ we have $\frac{\partial U(x)}{\partial x_i^k}\leq 0$.\\
(ii) For $k=1,\ldots,K_i-1$, we have $\actionikone \in BR_i(x_{-i})$ if and only if $\frac{\partial U(x)}{\partial x_i^k}=0$. In particular, combined with (i) this implies that $\actionikone \not\in BR_i(x_{-i}) \iff \frac{\partial U(x)}{\partial x_i^k}<0$.
\end{lemma}
\begin{proof}
(i) Differentiating \eqref{eq_potential_expanded_form2} we find that
\begin{align} \label{eq_potential_derivative}
\frac{\partial U(x)}{\partial x_i^{k}} = U(\actionikone,x_{-i}) - U(\actionione,x_{-i}).
\end{align}
(i)  Since $\actionione$ is a best response, we must have $U(\actionione,x_{-i}) \geq U(\actionikone,x_{-i})$ for any $k=1,\ldots,K_i-1$. Hence $\frac{\partial U(x)}{\partial x_i^{k}}\leq 0$.\\
(ii) Follows readily from \eqref{eq_potential_expanded_form2}.
\end{proof}

\begin{lemma} \label{lemma_eq_BR_positive_gradient}
Let $x \in X$. If $\actionik \in BR_i(x_{-i})$ then $\frac{\partial U(x)}{\partial x_i^{k}} \geq 0$.
\end{lemma}
\begin{proof}
The result follows readily from \eqref{eq_potential_derivative}.
\end{proof}

\begin{lemma} \label{lemma_carrier_vs_gradient}
Suppose $x^*$ is an equilibrium and $y_i^k \in \carr(x^*)$, $k\geq 2$. Then $\frac{\partial U(x^*)}{\partial x_i^k} = 0$.
\end{lemma}
\begin{proof}
Since $U$ is multilinear, $y_i^k$ must be a pure-strategy best response to $x_{-i}^*$. The result then follows from Lemma \ref{lemma_apx_BR_vs_gradient}.
\end{proof}

\begin{lemma} \label{lemma_g_derivative_finite}
There exists a $c>0$ such that $|\frac{\partial \tilde \calP_i^k(x)}{\partial x_j^\ell}| < c$, $i=1,\ldots,\tilde{N}$, $k=1,\ldots,\gamma_i-1$, $j=1,\ldots,N$, $\ell\geq \gamma_j$ for $x$ in a neighborhood of $x^*$.
\end{lemma}
\begin{proof}
Differentiating \eqref{def_calPtilde_map} we see that $\frac{\partial \tilde \calP_i^k(x)}{\partial x_j^\ell} = -\frac{\partial g_i^k(x_p)}{\partial x_j^\ell}$, $i=1,\ldots,\tilde{N}$, $k=1,\ldots,\gamma_i-1$, $j=1,\ldots,N$, $\ell\geq \gamma_j$, $x=(x_p,x_m)$.

By the definition of $g$ we have $F(x_p,g(x_p),u) = 0$ for all $x_p$ in a neighborhood of $x_p^*$. Hence,
\begin{align}
0 & = D_{x_p} F(x_p,g(x_p),u)\\
& = D_{x_p} F(x_p,x_m',u)\big\vert_{x_m' = g(x_p)} + D_{x_m} F(x_p,x_m,u) D_{x_p} g(x_p),
\end{align}
By \eqref{def_mixed_hessian} and \eqref{def_F} we see that $D_{x_m} F(x_p,x_m,u) = \vH(x)$. Since the equilibrium $x^*$ is assumed to be non-degenerate, $\vH(x^*)$ is invertible and the above implies that
$$
D_{x_p} g(x_p^*) = \vH(x^*)^{-1} D_{x_p} F(x_p^*,x_m^*).
$$
Using \eqref{def_F} and the multilinearity of $U$, one may readily verify that $D_{x_p}F(x_p^*,x_m^*,u)$ is entrywise finite. Since $g$ is continuously differentiable, it follows that each entry of
$$
\left(\frac{\partial \tilde \calP_i^k(x)}{\partial x_j^\ell}\right)_{\substack{i=1,\ldots,\tilde{N}, k=1,\ldots,\gamma_i-1\\ j=1,\ldots,N,~\ell\geq \gamma_j}}
= \left(-\frac{\partial g_i^k(x_p)}{\partial x_j^\ell}\right)_{\substack{i=1,\ldots,\tilde{N}, k=1,\ldots,\gamma_i-1\\ j=1,\ldots,N,~\ell\geq \gamma_j}} = -D_{x_p} g(x_p)
$$
is uniformly bounded for $x=(x_p,x_m)$ in a neighborhood of $x^*$.
\end{proof}

\begin{lemma} \label{lemma_gradient_inequality_apx}
Suppose $V:\R^n\rightarrow\R$ is twice differentiable.
Suppose $x^*$ is a critical point of $V$ and the Hessian of $V$ at $x^*$, denoted by $\vH(x^*)$, is invertible. Then there exists a constant $c$ such that $\|\nabla V(x)\| \geq c d(x^*,x)$ for all $x$ in a neighborhood of $x^*$.
\end{lemma}
\begin{proof}
Suppose the claim is false. Then for any $\epsilon > 0$ there exists a sequence $(x_k^\epsilon)_{k\geq 1}\subset B(x^*,\epsilon)$ such that $\|\nabla V(x_k)\| < \frac{1}{k} d(x_k,x^*)$. Let $(x_k)_{k\geq 1}$ be such a sequence that furthermore satisfies $\lim_{k\rightarrow\infty }d(x_k,x^*)=0$. Let $y_k\in\R^n$, $t_k\in\R$ be such that $x_k = x^*+t_k y_k$, $\|y_k\| = 1$.  Since $(y_k)_{k\geq 1}$ is a sequence on the unit sphere in $\R^n$ it has a convergent subsequence; say, $y_{k_j} \rightarrow y$ as $j\rightarrow\infty$.
Let $f:\R\rightarrow\R$ be given by $f(t):= V(x^*+ty)$.

Using the continuity of $\nabla V$ we see that for any $c>0$ we have $|f'(t)| < c t$ for all $t$ sufficiently small. Since $x^*$ is a critical point of $V$ we have $f'(0) = 0$. Hence
$$
f''(0) = \lim_{t\rightarrow 0} \bigg\vert \frac{f'(t) - f'(0)}{t} \bigg\vert = \lim_{t\rightarrow 0} \frac{|f'(t)|}{t} < c.
$$
Letting $c\rightarrow 0$ we see that $f''(0) = 0$. But this means $0=f''(0) = y^T \vH(x^*) y$, implying the Hessian is singular, which is a contradiction.
\end{proof}

The following lemma characterizes the level sets of polynomial functions. Before presenting the lemma we require the following definition.
\begin{definition}
Given a polynomial $p:\R^n\to\R$, $n\geq 1$, let
$$
Z(p) :=\{x\in\R^n:~p(x)=0\}
$$
be the zero-level set of $p$.
\end{definition}
\begin{lemma} \label{lemma_polynomial_level_set}
Let $p(x):\R^n\rightarrow \R$, $n\geq 1$ be a polynomial that is not identically zero. Then $\calL^n(Z(p)) = 0$.
\end{lemma}

\begin{proof}

We will prove the result using an inductive argument.

Suppose first that $n=1$ so that $p:\R\rightarrow \R$. Let $k$ denote the degree of $p$. Since $p$ is not identically zero, the fundamental theorem of algebra implies that $p$ has at most $k$ zeros. Hence $\calL^1(Z(p)) = 0$.

Now, suppose that $n\geq 2$ and for any polynomial $\tilde{p}:\R^{n-1} \to \R$ there holds $\calL^{n-1}(Z(\tilde{p})) = 0$. We may write
$$
p(x,x_n) = \sum_{j=0}^k p_j(x)x_n^j,
$$
where $k$ is the degree of $p$ in the variable $x_n$, $x=(x_1,\ldots,x_{n-1})$, the functions $p_j$, $j=0,\ldots,k$ are polynomials in $n-1$ variables, and where at least one $p_j$ is not identically zero.

If $(x,x_n)$ is such that $p(x,x_n)=0$ then there are two possibilities: Either (i) $p_0(x) =\ldots = p_k(x) = 0$, or (ii) $x_n$ is the root of the one-variable polynomial $p_x(t) := \sum_{j=1}^k p_j(x)t^j$.

Let $A$ and $B$ be the subsets of $\R^n$ where (i) and (ii) hold respectively, so that $Z(p) = A\cup B$.
For any $x_n\in \R$ we have $(x,x_n) \in A \iff x\in Z(p_j), ~\forall j=1,\ldots,k$. By the induction hypothesis, we have $\calL^{n-1}(Z(p_j)) = 0$ for at least one $j$, and hence $\int_{\R^{n-1}} \rchi_{A}(x,x_n) \dx = 0$ for any $x_n\in \R$, where we include the argument in the characteristic function $\rchi_A$, in order to emphasize the dependence on both $x$ and $x_n$. This implies that $x_n\mapsto \int_{\R^{n-1}} \rchi_{(x,x_n)\in A} \dx$ is a measurable function (it's identically zero) and
$$
\calL^n(A) = \int_{\R} \int_{\R^{n-1}} \rchi_{A}(x,x_n) \dx \,dx_n = 0.
$$
Now, by the fundamental theorem of algebra, for any $x\in \R^{n-1}$ there are at most $k$ values $t\in \R$ such that $(x,t)\in B$, and hence $\int_{\R} \rchi_{B}(x,x_n) \,dx_n = 0$. As before, this implies that $x\mapsto \int_{\R} \rchi_{B}(x,x_n) \,dx_n$ is a measurable function and
$$
\calL^n(B) = \int_{\R^{n-1}}\int_{\R} \rchi_{B}(x,x_n) \,dx_n \dx = 0.
$$
Since $Z(p) = B\cup A$, this proves the desired result.
\end{proof}

\begin{remark} \label{remark_level_sets}
Note that if $p\equiv 0$, then $Z(p) = \R^n$. Thus, in general, if $p:\R^n\to \R$ is a polynomial, then Lemma \ref{lemma_polynomial_level_set} implies that either $Z(p) = \R^n$ or $\calL^n(Z(p)) = 0$.
\end{remark}
\begin{lemma} \label{lemma_thin_indifference_surface}
Suppose $\Gamma$ is a non-degenerate potential game. Then each indifference surface $\calI_{i,k,\ell}$, as defined in \eqref{def_indiff_surface}, is a union of smooth surfaces with Hausdorff dimension at most $\kappa-1$.
\end{lemma}
\begin{proof}
Throughout the proof, when we refer to the dimension of a set we mean the Hausdorff dimension.
Let $i\in\{1,\ldots,N\}$, $k,\ell\in\{1,\ldots,K_i\}$, $k\not= \ell$ and let $\calI := \calI_{i,k,\ell}$, where $\calI_{i,k,\ell}$ is as defined in \eqref{def_indiff_surface}. Note that $\calI$ is the zero-level set of the polynomial $p(x) := U(y_i^k,x_{-i}) - U(y_i^\ell,x_{-i})$. By Lemma \ref{lemma_polynomial_level_set} and Remark \ref{remark_level_sets} we see that either $\calL^{\kappa}(\calI)=0$, or
$\calI=X$.
Being the level set of a polynomial, if $\calL^\kappa(\calI) = 0$, then $\calI$ is the union of smooth surfaces with dimension at most $\kappa-1$.

Suppose that $\calI$ has dimension greater than $\kappa-1$. Then by the above, we see that $\calI=X$. Since $\Gamma$ is a finite normal-form game, there exists at least one equilibrium $x^*\in X$. Letting $x^*$ be written componentwise as $x^* = ([x^*]_j^m)_{j=1,\ldots,N,~m=1,\ldots,K_i-1}$ we see that if  $[x^*]_i^k > 0$, then $x^* \in \calI= X$ implies that $x^*$ is a second-order degenerate equilibrium. Otherwise, if $[x^*]_i^k = 0$, then $x^* \in \calI = X$ implies that $x^*$ is a first-order degenerate equilibrium. In either case we see that $x^*$ is a degenerate equilibrium, and hence $\Gamma$ is a degenerate game, which is a contradiction.

Since $\calI$ was an arbitrary indifference surface, we see that if $\Gamma$ is a non-degenerate game, then every indifference surface has dimension at most $\kappa-1$.
\end{proof}

\begin{lemma} \label{lemma_ball_boundary}
Let $B_\epsilon$ be as defined in the proof of Proposition \ref{prop_finite_time_convergence}. Then,
$$
\calH^{\kappa-1}(\partial B_\epsilon) \rightarrow 0 \quad \mbox{ as } \quad \epsilon\to 0.
$$
\end{lemma}
\begin{proof}

Following standard notation (see \cite{EvansGariepy}, Chapter 2), for $0\leq s<\infty$, $0<\delta\leq \infty$, and $A\subset \R^\kappa$, let
$$
\calH^{s}_{\delta}(A):= \inf\bigg\{\sum_{j=1}^\infty \alpha(s)\left(\frac{\diam C_j}{2}\right)^s:~ A\subset\bigcup_{j=1}^\infty C_j, ~\diam C_j \leq \delta \bigg\},
$$
where $\alpha(s):= \frac{\pi^s}{\Gamma(\frac{s}{2})+1}$, and where $\Gamma$ in this context denotes the $\Gamma$ function

By our construction of $(B_\epsilon^j)_{j\geq 1}$, for every $\epsilon > 0$ we have
$$\sum_{j=1}^\infty \alpha(s) \left(\frac{\diam B_\epsilon^j}{2}\right)^{\kappa-2} < 2\calH^{\kappa-2}(Q\cup\partial \calZ) <\infty.$$
Since $\diam B_\epsilon^j \leq \epsilon$ for every $\epsilon>0$, $j\in\N$, this gives
\begin{align*}
\lim_{\epsilon\to 0} \calH^{\kappa-1}_\epsilon(\partial B_\epsilon) & \leq \lim_{\epsilon\to 0} \calH^{\kappa-1}_\epsilon(B_\epsilon)\\
 & \leq \lim_{\epsilon\to 0}\sum_{j=1}^\infty \alpha(s) \left(\frac{\diam B_\epsilon^j}{2}\right)^{\kappa-1}\\
& \leq \lim_{\epsilon\to 0} \epsilon \sum_{j=1}^\infty \alpha(s) \left(\frac{\diam B_\epsilon^j}{2}\right)^{\kappa-2}\\
& \leq \lim_{\epsilon \to 0} \epsilon 2\calH^{\kappa-2}(Q\cup\partial \calZ)~ = 0.
\end{align*}
By the definition of the Hausdorff measure we have $\calH^{\kappa-1}(\partial B_\epsilon) := \sup_{\delta > 0}\calH^{\kappa-1}_\delta(\partial B_\epsilon)$. Hence, the above implies $\lim_{\epsilon\to 0} \calH^{\kappa-1}(\partial B_\epsilon) = 0$.
\end{proof}

\begin{lemma} \label{lemma_clQ_dim}
Let $Q$ be defined as in Section \ref{sec_finite_time_convergence}. Then $\cl Q$ has Hausdorff dimension at most $\kappa-2$.
\end{lemma}
\begin{proof}
Let $A$ be the subset of $X$ where two or more decision surfaces intersect.
Let $N\subset A$ be the subset of $X$ where two or more decision surfaces intersect and their normal vectors coincide.
Define the \emph{relative interior} of $N$ with respect to $A$ as
$$
\ri N := \{x\in N:~ \exists \epsilon>0 \mbox{ s.t. }  B(x,\epsilon)\cap A \subset N\},
$$
and define the \emph{relative boundary} of $N$ with respect to $A$ as
$$
\partial N := \cl N \backslash \ri N.
$$
Since each indifference surface has Hausdorff dimension $\kappa-1$, $N$ has Hausdorff dimension at most $\kappa-1$. In particular, $N$ is the union of a finite number of smooth $\kappa-1$ dimensional surfaces and a component with Hausdorff dimension at most $\kappa-2$. This implies that the relative boundary of $N$ has Hausdorff dimension at most $\kappa-2$.

Let $\tilde Q$ be as defined in Section \ref{sec_finite_time_convergence}.
Note that the closure of $\tilde Q$ satisfies $\cl \tilde Q \subseteq \tilde Q \cup \partial N$. Since the sets $\tilde Q$ and $\partial N$ have Hausdorff dimension at most $\kappa-2$, the set $\cl \tilde Q$ also has Hausdorff dimension at most $\kappa-2$.

Let $\Lambda(x^*)$ be as defined in Section \ref{sec_proof_main_result}. If $\Lambda(x^*) = \{x^*\}$, then $\Lambda(x^*)$ is closed and has Hausdorff dimension 0. Otherwise, $\Lambda(x^*)$ is defined as the graph of $g$. In Section \ref{sec_diff_ineq_prelims} it was shown that $\Gr(g)$ has Hausdorff dimension at most $\kappa-2$. Since $g$ is a smooth function, the closure of $\Gr(g)$ has Hausdorff dimension at most $\kappa-2$.

Recall that $Q$ is defined as $Q=\tilde{Q}\cup \Lambda(x^*)$ and hence $\cl Q = \cl \tilde Q \cup \cl \Lambda(x^*)$. Since $\cl \tilde Q$ and $\cl \Lambda(x^*)$ each have Hausdorff dimension at most $\kappa-2$, $\cl Q$ also has Hausdorff dimension at most $\kappa-2$.
\end{proof}

\begin{lemma}\label{lemma_calZ_normal}
Let $\calZ$ be as defined \eqref{def_calZ}. Then for any $x\in \calZ$ there holds $\nu\cdot y = 0$ for any vector $\nu$ normal to $\calZ$ at $x$, and any $y\in \FP(x)$
\end{lemma}
\begin{proof}
Suppose $x\in X$ and $x$ is in some indifference surface $\calI_{i,k,\ell}$. Suppose $\nu$ is a vector that is normal to $\calI_{i,k,\ell}$ at $x$.
By the definition of $\calI_{i,k,\ell}$, if $x\in \calI_{i,k,\ell}$ then for all $\hat x\in X$ such that $\hat x_{-i} = x_{-i}$ we have $\hat{x}\in \calI_{i,k,\ell}$. This implies that the $(i,\tilde{k})$-th component of $\nu$ must be zero for every $\tilde{k} = 1,\ldots,K_i-1$.

For $x\in X$, let $\calN(x):=\{(i,k):~ i\in\{1,\dots,N\},~k\in\{1,\ldots,K_i-1\},~ x\in\calI_{i,k,\ell} \mbox{ for some } ~\ell=1,\ldots,K_i-1,~\ell\not=k\}$ so that, given a point $x\in X$,  $\calN(x)$ specifies the indifference surfaces in which $x$ lies. Letting $\FP_i^k$ be the $(i,k)$-th component map of $\FP$, note that by the definition of an indifference surface, $\FP_i^k(x)$ is single valued for every pair $(i,k)\notin \calN(x)$.

Suppose $x\in X\backslash Q$ is in at least one decision surface $\calI$ and let $\nu$ be a vector that is normal to $\calI$ at $x$. Note that $x\notin Q$ implies that if $x$ is contained in any other decision surface $\hat \calI\not= \calI$, then $\nu$ is also normal to $\hat\calI$ at $x$. Letting $\nu$ be written componentwise as $\nu=(\nu_i^k)_{i=1,\ldots N,~k=1,\ldots,K_i-1}$, the above discussion implies that $\nu_i^k = 0$ for every pair $(i,k) \in \calN(x)$.

Now suppose $x\in \calZ$. By the definition of $\calZ$ we have $x\notin Q$ and $x$ is in at least one decision surface $\calI$. Let $\nu$ be a vector that is normal to $\calI$ at $x$.
By the definition of $\calZ$, there exists some $y\in \FP(x)$ such that $y\cdot \nu = 0$. Breaking this down in terms of components in $\calN(x)$ we have
$$
0 = y\cdot \nu = \sum_{(i,k)\in\calN(x)} y_i^k\nu_i^k + \sum_{(i,k)\notin\calN(x)} y_i^k\nu_i^k.
$$
The first sum is zero since $\nu_i^k = 0$ for all $(i,k) \in \calN(x)$. Consequently, the second sum must also be zero. But we have shown above that $F_i^k(x)$ is single valued for any $(i,k) \notin\calN(x)$. Hence, for any $\tilde{y}\in \FP(x)$ we have $\tilde{y}_i^k = y_i^k$ for all $(i,k) \notin\calN(x)$, and in particular,
$\sum_{(i,k)\notin\calN(x)} \tilde{y}_i^k\nu_i^k = \sum_{(i,k)\notin\calN(x)} y_i^k\nu_i^k = 0$. Moreover, since $\nu_i^k = 0$ for all $(i,k) \in\calN(x)$ we have $\sum_{(i,k)\in\calN(x)} \tilde{y}_i^k\nu_i^k=0$, which implies
$$
\tilde{y}\cdot \nu = \sum_{(i,k)\in\calN(x)} \tilde{y}_i^k\nu_i^k + \sum_{(i,k)\notin\calN(x)} \tilde{y}_i^k\nu_i^k = 0.
$$
Since $\tilde{y}\in \FP(x)$ was arbitrary, this proves the desired result.
\end{proof}

\section{Conclusions} \label{sec_conclusion}
The best-response dynamics \eqref{def_FP_autonomous} underlie many learning processes in game theory. We have shown that in any regular potential game (and hence, in almost every potential game \cite{swenson2017regular}), for almost every initial condition, the best-response dynamics \eqref{def_FP_autonomous} are well posed (i.e., there exists a unique solution) and converge to a pure-strategy NE. As a simple application of this result, we showed that solutions of \eqref{def_FP_autonomous} almost always converge at an exponential rate in potential games.

\bibliographystyle{siamplain}
\bibliography{myRefs}
\end{document}